\documentclass[12pt,reqno]{amsart}
\usepackage[margin=1in]{geometry}
\usepackage{amsmath,amssymb,amsthm,graphicx,amsxtra, setspace}
\usepackage[utf8]{inputenc}
\usepackage{mathrsfs}
\usepackage{hyperref}
\usepackage{xcolor}
\usepackage{upgreek}
\usepackage{mathtools}
\allowdisplaybreaks
\newtheorem{theorem}{Theorem}[section]
\newtheorem{lem}[theorem]{Lemma}
\newtheorem{rem}[theorem]{Remark}

\newtheorem{Def}[theorem]{Definition}

\newtheorem{Ex}[theorem]{Example}
\newtheorem{Ass}[theorem]{Assumption}
\usepackage{latexsym}
\usepackage{hyperref}
\usepackage{graphicx}
\usepackage{epstopdf}

\DeclareMathOperator*{\esssup}{ess\,sup}

\allowdisplaybreaks

\let\originalleft\left
\let\originalright\right
\renewcommand{\left}{\mathopen{}\mathclose\bgroup\originalleft}
\renewcommand{\right}{\aftergroup\egroup\originalright}

\newcommand{\Addresses}{{
		\footnote{

			\noindent	 \textsuperscript{1,3}Department of Applied Sciences and Engineering, Indian Institute of Technology Roorkee-IIT Roorkee, Roorkee, 247667, India.

			\noindent  \textit{e-mail\textsuperscript{1}:} \texttt{sumit@as.iitr.ac.in.}
			
			\noindent  \textit{e-mail\textsuperscript{3}:} \texttt{jay.dabas@gmail.com.}

			\noindent \textsuperscript{2}Department of Mathematics, Indian Institute of Technology Roorkee-IIT Roorkee, Roorkee, 247667, India.\par\nopagebreak
			\noindent  \textit{e-mail\textsuperscript{2}:} \texttt{maniltmohan@ma.iitr.ac.in, maniltmohan@gmail.com.}

			\noindent \textsuperscript{*}Corresponding author.

			\textit{Key words:} approximate controllability, non-instantaneous impulses, state-dpendent delay, fractional derivative, Mainardi’s Wright-type function.
			
			Mathematics Subject Classification (2020): 34K06, 34A12, 37L05, 93B05.

}}}

\begin{document}
	\title[Approximate controllability of fractional order impulsive systems]{Approximate controllability of non-instantaneous impulsive fractional evolution equations of order $1<\alpha<2$ with state-dependent delay in Banach spaces\Addresses}
	\author [S. Arora, M. T. Mohan and J. Dabas]{S. Arora\textsuperscript{1}, Manil T. Mohan\textsuperscript{2}  and J. dabas\textsuperscript{3*}}
	\maketitle{}
	\begin{abstract} The current article examines the approximate controllability problem for non-instantaneous impulsive fractional evolution equations of order $1<\alpha<2$ with state-dependent delay in separable reflexive Banach spaces. In order to establish sufficient conditions for the approximate controllability of our problem, we first formulate the linear-regulator problem and obtain the optimal control in feedback form. By using this optimal control, we deduce the approximate controllability of the linear fractional control system of order $1<\alpha<2$. Further, we derive sufficient conditions for the approximate controllability of the nonlinear problem. Finally, we provide a concrete example to validate the efficiency of the derived results .
	\end{abstract}
	\section{Introduction}\label{intro}\setcounter{equation}{0}
	
	Fractional evolution equations involve the derivatives of the form $\frac{d^{\alpha}}{dt^{\alpha}}$, where $\alpha>0$ is not necessarily an integer, have received much attention from researchers. This rising interest is motivated both by important applications of the theory, and by the difficulties involved in the mathematical structure. Fractional evolution equations appear in many physical phenomena arising from various scientific fields including analysis of viscoelastic materials, electrical engineering, control theory of dynamical systems, electrodynamics with memory, quantum mechanics, heat conduction in materials with memory, signal processing, economics, and many other fields (cf. \cite{HlR,AAK,YAR,IPF,MST}, etc). For an extensive study of the fractional differential equations (FDEs), we refer the interested readers to \cite{AAK,IPF}, etc. 
	
	The real-world phenomena such as natural disasters, harvesting, shocks, etc, experience unpredictable changes in their states for negligible time duration. These sudden changes are relevant and interesting, and well approximated by instantaneous impulses. Such processes are conveniently described by instantaneous impulsive evolution equations. The theory of instantaneous impulsive evolution equations has found various applications such as optimal control models in economics, threshold, pharmacokinetics, bursting rhythm models in medicine and frequency modulated systems etc (cf. \cite{BM,LVB,NE,TY} etc). Moreover, the evolution processes in pharmacotherapy, such as the consequent absorption of the body, in hemodynamical steadiness of a person and the insulin distribution in the bloodstream are moderate and continuous. Hence, the dynamics of such phenomena cannot be interpreted by the instantaneous impulsive systems.  Therefore, this situation can be appropriately estimated as an impulsive action that starts suddenly and continuously active over a finite time interval. Thus, the dynamics of such processes are suitably described by non-instantaneous impulsive systems. 
	
	Hern\'andez et. al. \cite{EHD}, first introduced a class of impulses that are suddenly active for an arbitrary fixed point and remains in action on a finite time interval.  Such impulses are called non-instantaneous impulses. Then, they developed the existence of solutions for non-instantaneous impulsive differential equations. Later, Wang et. al. \cite{JRW,JRY} extended this model to two general classes of impulsive differential equations, which are very useful to characterize certain dynamics of evolutionary processes in pharmacotherapy. For the study of   non-instantaneous impulsive systems, one can refer  to \cite{JMF,YTJ}, etc, and the references therein.
	
	On the other hand, there are several real-world phenomena in which the current state of a system depends on the past states. The dynamics of such types of processes are suitably analyzed by delay (functional) evolution system, which naturally occurs in logistic reaction-diffusion processes with delay, ecological models, neural network, inferred grinding models, etc (cf. \cite{FC,LA,NJ}, etc). Moreover, the evolution equations with state-dependent delays are more frequent and widely applicable, see for example \cite{AO,FC}, etc, and the references therein.

	The notion of controllability is one of the fundamental concepts in mathematical control theory and play a significant role in many control problems such as, irreducibility of transition semigroups \cite{GDJZ}, the optimal control problems \cite{VB}, stabilization of unstable systems via feedback control \cite{VB1} etc. Controllability (exact and approximate) means that the solution of a considered dynamical control system can steer from an arbitrary initial state to a desired final state by using some suitable control function. For the infinite dimensional control systems, the problem of approximate controllability is more substantial and is having broad range of applications, see for instance, \cite{M,TRR,TR,EZ}, etc. In the past few decades, the problem of approximate controllability for different kinds of dynamical systems (in Hilbert and Banach spaces) such as fractional evolution equations, Sobolev type systems, delay (functional) differential equations, impulsive systems, stochastic differential equations, etc, has been widely studied by many researchers and they have produced excellent results, see for instance, \cite{SAM,SM,AGK,M,ER,Z} etc, and the references therein. 
	
	It is well known that the slow diffusion processes are described by first-order evolution equations while the fast diffusion processes are characterized by second-order evolution systems. Similarly, FDEs of order $\alpha\in(0,1)$ stand for slow anomalous diffusion processes (subdiffusion process), while the fast anomalous diffusion processes (superdiffusion process) are suitably modeled by FDEs of order $\alpha>1$. Therefore, by this reason, the theory for fractional evolution equations of order $\alpha\in(0,1)$ is discrete from that for fractional evolution equations of order $\alpha>1$, and it inspired the researchers to study both theories separately (see  \cite{Jkr,Jkj}, etc for more details). In the past two decades, a good number of results came out on the existence of solutions and approximate controllability for fractional systems of order $\alpha\in(0,1)$ in Hilbert and Banach spaces, see for example, \cite{DNC,JA,GR,SNS,MIN,NMI,JWY,Z}, etc, and the references therein. Recently, many authors established the existence of  mild solutions and approximate controllability for  fractional evolution equations of order $\alpha\in(1,2)$ by invoking the theory of sectorial operators via suitable fixed point theorems, see for instance, \cite{Qh,Sr,Sbl,Sb} etc. 
	
	Recently, a few developments have been reported on the existence and controllability of fractional evolution equations of order $1<\alpha<2$  by applying the cosine and sine family of bounded linear operators (cf. \cite{Jh,MMr,Yz} etc). A new concept of mild solutions for fractional evolution systems with order $\alpha\in (1, 2)$ in Banach spaces  by using  the Laplace transform and Mainardi's Wright-type function is introduced in \cite{Yz}. {Henr\'iquez et.al \cite{JFA} investigated the existence of a classical solution for the fractional homogeneous and inhomogeneous abstract Cauchy problems of order $\alpha\in(1,2)$ by using the $\alpha$-resolvent family corresponding to the homogeneous Cauchy problem}. Raja et. al. \cite{MMv} formulated sufficient conditions for the approximate controllability of  fractional differential equation of order $1<\alpha<2$ in Hilbert spaces by applying the concept of cosine and sine family of operators via Schauder's fixed point theorem. Raja and Vijaykumar \cite{Mvv} derived sufficient conditions for the approximate controllability of  fractional integro-differential system of order $\alpha\in(1,2)$ in Hilbert spaces. Later, Vijaykumar et.al. in \cite{Vcr} demonstrated the approximate controllability of Sobolev-type Volterra-Fredholm integro-differential equations of order $1<\alpha<2$ in Hilbert spaces by invoking the Krasnoselskii fixed point theorem. 
	
	To the best of our knowledge, the approximate controllability for impulsive fractional differential equations of order $1<\alpha<2$ with state-dependent delay in Banach spaces by using the idea of cosine and sine family of operators is not investigated yet, and it motivated us to encounter this problem. In this paper, we consider semilinear impulsive fractional differential equations of order $1<\alpha<2$ with non-instantaneous impulses and state-dependent delay in separable reflexive Banach spaces. We derive sufficient conditions for the approximate controllability of the considered system by applying the concept of $\alpha/2$-resolvent family (defined in \eqref{2.3}, which is related to the strongly continuous consine family generated by the operator $\mathrm{A}$), resolvent operator condition and Schauder's fixed point theorem. Moreover, we discuss the construction of a feedback control and the approximate controllability of the linear fractional control system of order $1<\alpha<2$ in details (see, Section \ref{sec3}), which lacks in the available literature. Furthermore, due to the observation related to characterization of the phase space for impulsive systems discussed in \cite{GU}, we replace the uniform norm on the phase space by the integral norm (see Example \ref{exm2.8}),

	Let us consider the following non-instantaneous impulsive fractional order system with state-dependent delay: 
	\begin{equation}\label{1.1}
		\left\{
		\begin{aligned}
			^C\mathrm{D}_{0,t}^{\alpha}x(t)&=\mathrm{A}x(t)+\mathrm{B}u(t)+f(t,x_{\varrho(t, x_t)}), \ t\in\bigcup_{j=0}^{p} (s_j, \tau_{j+1}]\subset J=[0,T], \\
			x(t)&=h_j(t, x(\tau_j^{-})),\ t\in(\tau_j, s_j], \ j=1,\dots,p, \\
			x'(t)&=h'_j(t, x(\tau_j^{-})),\ t\in(\tau_j, s_j], \ j=1,\dots,p, \\
			x_{0}&=\psi\in \mathfrak{B}, x'(0)=\eta,
		\end{aligned}
		\right.
	\end{equation}
	where \begin{itemize} \item $^C\mathrm{D}_{0,t}^{\alpha}$ is the Caputo fractional derivative of order $\alpha$ with $1<\alpha<2$, \item  the operator $\mathrm{A}:\mathrm{D(A)}\subset \mathbb{X}\to\mathbb{X}$ is linear (not necessarily bounded), where $ \mathbb{X}$ is a separable reflexive Banach space having a strictly convex dual $\mathbb{X}^*$, \item the linear operator $\mathrm{B}:\mathbb{U}\to\mathbb{X}$ is bounded with $\left\|\mathrm{B}\right\|_{\mathcal{L}(\mathbb{U};\mathbb{X})}= \tilde{M}$ and the control function $u\in \mathrm{L}^{2}(J;\mathbb{U})$ (admissible control class), where $\mathbb{U}$ is a separable Hilbert space, \item  the appropriate function $f:J\times \mathfrak{B}\rightarrow \mathbb{X},$  where $\mathfrak{B}$ is a phase space satisfying some axioms which will be provided  in next section, \item the functions $h_j:[\tau_j, s_j]\times\mathbb{X}\to\mathbb{X}$ and their derivatives $h'_j:[\tau_j, s_j]\times\mathbb{X}\to\mathbb{X}$ for $j=1,\ldots,p$, denote the non-instantaneous impulses and the fixed points $s_j$ and $\tau_j$ satisfy $0=\tau_0=s_0<\tau_1\le s_1\le \tau_2<\ldots<\tau_p\le s_p\le \tau_{p+1}=T$,   the values  $x(\tau_j^+)$ and $x(\tau_j^-)$ stand for the right and left limit of $x(t)$ at the point $t=\tau_j,$ respectively and satisfy $x(\tau_j^-)=x(\tau_j),$ for $j=1,\ldots,p$, \item the function $x_{t}:(-\infty, 0]\rightarrow\mathbb{X}$ with $x_{t}(\theta)=x(t+\theta)$ and  $x_t\in \mathfrak{B}$, for each $t\in J$. The function $\varrho: J \times \mathfrak{B} \rightarrow (-\infty, T]$ is continuous.
	\end{itemize}
	
	The rest of the article is structured as follows: In section \ref{sec2}, we recall some basic definitions of fractional calculus, property of cosine family and the phase space axioms. Moreover, we also provide the assumptions and important results which is useful in the subsequent sections. Section \ref{sec3} is reserved for the approximate controllability of the linear fractional control system of order $1<\alpha<2$. For this, first we formulate the linear regulator problem, show the existence of an optimal control (Theorem \ref{th1}), and then we derive an explicit form of the optimal control in the feedback form in Lemma \ref{lem3.1}. Further, we examine the approximate controllability of the linear fractional control system \eqref{3.2}, by using the optimal control (Lemma \ref{lem3.4}). In section \ref{semilinear}, we first prove the existence of a mild solution for the control system \eqref{1.1} by applying Schauder's fixed point theorem (Theorem \ref{thm4.3}), and then demonstrate the approximate controllability of the system \eqref{1.1} in Theorem \ref{thm4.4}. In the final section, we present an example to support the application of the developed theory.
	
	\section{Preliminaries}\label{sec2}\setcounter{equation}{0}
	In the present section, we introduce some standard notations, basic definitions and facts about fractional calculus. Moreover, we also recall some important properties of cosine and sine families, and also provide the axioms of phase space. 
	
	Let $\mathrm{AC}([a,b];\mathbb{R})$ denote the space of all mappings from $[a,b]$ to $\mathbb{R},$ which is absolutely continuous and $\mathrm{AC}^n([a,b];\mathbb{R}),$ for $n\in\mathbb{N},$ represent the space of all functions $f:[a,b]\to\mathbb{R}$, which have continuous derivatives up to order $n-1$  with $f^{(n-1)}\in\mathrm{AC}([a,b];\mathbb{R})$.

	First, we recall some basic definitions from fractional calculus and the details can be found in \cite{AAK,IPF}, etc.
	\begin{Def}
		The \emph{fractional integral} of order $\beta>0$ for a function $f\in\mathrm{L}^1([a,b];\mathbb{R})$, $a,b\in\mathbb{R}$ with $a<b$, is given as
		\begin{align*}
			I_{a}^{\beta}f(t):=\frac{1}{\Gamma(\beta)}\int_{a}^{t}\frac{f(s)}{(t-s)^{1-\beta}}\mathrm{d}s,\ \mbox{ for a.e. } \  t\in[a,b],
		\end{align*}
		where $\Gamma(\alpha)=\int_{0}^{\infty}t^{\alpha-1}e^{-t}\mathrm{d}t$ is the Euler gamma function.
	\end{Def}
	\begin{Def}
		The \emph{Riemann-Liouville fractional derivative} of order $\beta>0$ for a function  $f\in\mathrm{AC}^n([a,b];\mathbb{R})$  is defined as 
		\begin{align*}
			^L\mathrm{D}_{a,t}^{\beta}f(t):=\frac{1}{\Gamma(n-\beta)}\frac{d^n}{dt^n}\int_{a}^{t}(t-s)^{n-\beta-1}f(s)\mathrm{d}s,\ \mbox{ for a.e. }\ t\in[a,b],
		\end{align*}
		where $n-1< \beta<n$.
	\end{Def}

	\begin{Def}
		The \emph{Caputo fractional derivative} of a function $f\in\mathrm{AC}^n([a,b];\mathbb{R})$ of order $\beta>0$   is defined as 
		\begin{align*}
			^C\mathrm{D}_{a,t}^{\beta}f(t):=\ ^L\mathrm{D}_{a,t}^{\beta}\left[f(t)-\sum_{k=1}^{n-1}\frac{f^{(k)}(a)}{k!}(x-a)^k\right],\ \mbox{ for a.e. } \ t\in[a,b].
		\end{align*}
	\end{Def}
	\begin{rem}[\cite{IPF}]
		If a function $f\in\mathrm{AC}^n([a,b];\mathbb{R})$, then the Caputo fractional derivative of $f$ can be written as
		\begin{align*}
			^C\mathrm{D}_{a,t}^{\beta}f(t)=\frac{1}{\Gamma(n-\beta)}\int_{a}^{t}(t-s)^{n-\beta-1}f^{(n)}(s)\mathrm{d}s,\ \mbox{for a.e.}\ t \in[a,b], \ n-1< \beta<n.
		\end{align*}
	\end{rem}
	\subsection{The strongly continuous cosine family:} In this subsection, we first construct a strongly continuous cosine family $\left\{\mathrm{C}\left(t \right) :t\in \mathbb{R}\right\}$ in $\mathbb{X}$ generated by $\mathrm{A}$, and then review some important properties of the cosine and sine families. In order to construct a strongly continuous cosine family $\left\{\mathrm{C}\left(t \right) :t\in \mathbb{R}\right\}$, we impose the following assumptions on the linear operator $\mathrm{A}:\mathrm{D}(\mathrm{A})\to\mathbb{X}$.
	\begin{Ass}\label{ass2.1}
		The operator $\mathrm{A}$ fulfills the following conditions:
		\begin{itemize}
			\item[(R1)]The linear operator $\mathrm{A}$ is closed with dense domain $\mathrm{D}(\mathrm{A} )$ in $\mathbb{X}$.
			\item[(R2)] For real $\lambda$, $\lambda>0,$ $\lambda^2$ is in the resolvent set $\rho(\mathrm{A})$ of $\mathrm{A}$.
			The resolvent $\mathrm{R}(\lambda^2; \mathrm{A})$ exists, strongly infinitely differentiable, and satisfies 
			$$\left\|\left(\frac{\mathrm{d}}{\mathrm{d}\lambda}\right )^k(\lambda \mathrm{R}(\lambda^2; \mathrm{A}))\right\|_{\mathcal{L}(\mathbb{X})}\!\!\!\le\frac{Mk!}{\lambda^{k+1}},$$
			for $k=0,1,2,\dots. $
			\item[(R3)]  $\mathrm{R}(\lambda^2;\mathrm{A})$ is compact for some $\lambda$ with $\lambda^2\in \rho(\mathrm{A}).$
		\end{itemize}
	\end{Ass}
	
	\begin{Def}[\cite{TCC1}]\label{cosine}
		A family of bounded linear operators $\{\mathrm{C}(t):t\in \mathbb{R}\}$ in a Banach space $\mathbb{X}$ is called \emph{a strongly continuous cosine family} if
		\begin{enumerate}
			\item [(1)] $\mathrm{C}(s+t)+\mathrm{C}(s-t)=2\mathrm{C}(s)\mathrm{C}(t)$, for $ s, t \in \mathbb{R}$;
			\item [(2)] $\mathrm{C}(0)=\mathrm{I}$, where $\mathrm{I}$ represents the identity operator;
			\item[(3)] $\mathrm{C}(t)x$ is continuous in t on $\mathbb{R},$ for each fixed $x\in\mathbb{X}$.
		\end{enumerate}
	\end{Def}
	The family $\{\mathrm{S}(t):t\in \mathbb{R}\}$ on $\mathbb{X}$  defined by 
	\begin{align}
		\nonumber \mathrm{S}(t)x&=\int_{0}^{t}\mathrm{C}(s)x\mathrm{d}s,\ x\in \mathbb{X},\ t\in\mathbb{R}.
	\end{align}
	is called \emph{strongly continuous sine family}.
	
	The infinitesimal generator $\mathrm{A}$ of a strongly continuous cosine family $\{\mathrm{C}(t):t\in \mathbb{R}\}$ is defined by
	$$\mathrm{A}x=\frac{\mathrm{d}^2}{\mathrm{d}t^2}\mathrm{C}(t)x\Big|_{t=0},\ x\in \mathrm{D}(\mathrm{A}),$$  where 
	\begin{align*}
		\mathrm{D}(\mathrm{A})&= \Big\{ x\in \mathbb{X}:\mathrm{C}(t)x \text{ is two times continuously differentiable function} \text{ with respect to } t\Big\},
	\end{align*} 
	equipped with the graph norm $$\left\|x\right\|_{\mathrm{D}(\mathrm{A})}=\left\|x\right\|_\mathbb{X}+\left\|\mathrm{A}x\right\|_\mathbb{X},\ x\in \mathrm{D}(\mathrm{A}).$$ We also define the set 
	$$\mathrm{E}=\Big\{x: \mathrm{C}(t)x \text{ is continuously differentiable function with respect to}\ t\Big\},$$ with the norm
	$$\left\|x\right\|_1=\left\|x\right\|_\mathbb{X}+\sup_{0\le t\le 1}\left\|\mathrm{AS}(t)x\right\|_{\mathbb{X}},\  x\in \mathrm{E},$$
	forms a Banach space, see \cite{JKI}.

	\begin{lem}[Proposition 2.7, \cite{TCC2}]\label{lem2.1}
		If the assumptions (R1)-(R2) hold true. Then the operator $\mathrm{A}$ generates a strongly continuous cosine family $\{\mathrm{C}(t):t\in \mathbb{R}\}$, which is uniformly bounded, that is, there exists a constant $M\ge1$ such that $\left\|\mathrm{C}(t)\right\|_{\mathcal{L}(\mathbb{X})}\le M$.
	\end{lem}
	\begin{lem}[Proposition 5.2, \cite{TCC2}]\label{lem2.2}
		Suppose that the operator $\mathrm{A}$ generates a strongly continuous cosine family $\{\mathrm{C}(t):t\in\mathbb{R}\}$ on $\mathbb{X}$. Let  $\{\mathrm{S}(t):t\in \mathbb{R}\}$ be the associated sine family on $\mathbb{X}$. Then the following are equivalent:
		\begin{itemize}
			\item [(i)] The operator $\mathrm{S}(t)$ is compact for every $t\in\mathbb{R}$.
			\item [(ii)] Assumption (R3) holds.
		\end{itemize}
	\end{lem}

	\begin{lem}[Propositions 2.1 and 2.2, \cite{TCC1}]\label{lem2.3}
		Let the operator $\mathrm{A}$ generates a strongly continuous cosine family $\{\mathrm{C}(t):t\in\mathbb{R}\}$, which is uniformly bounded in $\mathbb{X}$, and let $\{\mathrm{S}(t):t\in \mathbb{R}\}$ be the associated sine family on $\mathbb{X}$. Then the  following are true:
		\begin{itemize}
			\item [(i)] $\left\|\mathrm{S}(t_2)-\mathrm{S}(t_1)\right\|_{\mathcal{L}(\mathbb{X})}\le M\left|t_2-t_1\right|,$ for all $t_1,t_2\in\mathbb{R}$;
			\item [(ii)]  If $x\in\mathrm{E}$, then $\mathrm{S}(t)x\in\mathrm{D}(\mathrm{A})$ and $\frac{d}{dt}\mathrm{C}(t)x=\mathrm{A}\mathrm{S}(t)x$.
			
		\end{itemize}
	\end{lem}
	\subsection{Phase space:} Let us provide the definition of the phase space $\mathfrak{B}$ in axiomatic form (cf. \cite{HY}) and appropriately modify for the impulsive systems (cf. \cite{VOj}). Specifically, the phase space $\mathfrak{B}$ is a linear space of all mapping from $(-\infty, 0]$ into $\mathbb{X}$ equipped with the seminorm $\left\|\cdot\right\|_{\mathfrak{B}}$ and satisfying the following axioms:
	\begin{enumerate}
		\item [(A1)] If $x: (-\infty, \nu+\sigma)\rightarrow \mathbb{X},\  \sigma>0,$ such that $x_{\nu}\in \mathfrak{B}$ and $x|_{[\nu, \nu+\sigma]}\in \mathrm{PC}([\nu, \nu+\sigma];\mathbb{X}),$ then for every $t\in [\nu, \nu+\sigma],$ the following conditions hold true:
		\begin{itemize}
			\item [(i)] $x_{t}$ is in $\mathfrak{B}$;
			\item [(ii)] $\left\|x_{t}\right\|_{\mathfrak{B}}\leq \mathcal{P}(t-\nu)\sup\{\left\|x(s)\right\|_{\mathbb{X}}: \nu 
			\leq s\leq t\}+\mathcal{Q}(t-\nu)\left\|x_{\nu}\right\|_{\mathfrak{B}},$ where $\mathcal{P},\ \mathcal{Q}:[0, \infty)\rightarrow [0, \infty)$ such that $\mathcal{P}$ is continuous, $\mathcal{Q}$ is locally bounded, and both $\mathcal{P},\mathcal{Q}$ are independent of $x(\cdot).$
		\end{itemize}
		\item [(A2)] The space $\mathfrak{B}$ is complete.
	\end{enumerate}  
	For any $\psi\in \mathfrak{B},$ the function $\psi_{t}, \ t\leq 0,$ is defined by  $\psi_{t}(\theta)=\psi(t+\theta),\  \theta \in (-\infty, 0].$ If the function $x(\cdot)$ fulfills the axiom (A1) with $x_{0}=\psi,$ then we can extend the mapping $t\mapsto x_{t}$ by taking $x_{t}=\psi_{t}, \  t\leq0$, to the whole interval $(-\infty, T].$ Moreover, let us define a set 
	\begin{align*}
		\mathfrak{R}(\varrho^{-})&=\{\varrho(s, \psi):\  \varrho(s, \psi)\leq 0, \ \text{ for }\  (s, \psi)\in J \times \mathfrak{B}\},
	\end{align*} 
	where the function $\varrho$ is  the same as the one defined in section \ref{intro}. Further, we assume that the function $\psi_{t}:\mathfrak{R}(\varrho^{-})\to\mathfrak{B}$ is continuous in $t$, and there exists a continuous and bounded function $\varTheta^{\psi}: \mathfrak{R}(\varrho^{-})\rightarrow (0, \infty)$ such that  
	\begin{align*}
		\left\|\psi_{t}\right\|_{\mathfrak{B}}&\leq \varTheta^{\psi}(t)\left\|\psi\right\|_{\mathfrak{B}}.
	\end{align*}
	Let us define the set 
	\begin{align*} 
		\mathrm{PC}(J;\mathbb{X})&:=\Big\{x:J \rightarrow \mathbb{X} : x\vert_{t\in I_i}\in\mathrm{C}(I_i;\mathbb{X}),\ I_i:=(\tau_i, \tau_{i+1}],\ i=0,1,\ldots,p,\\& \qquad  x(\tau_i^+) \mbox{ and }\ x(\tau_i^-) \mbox{ exist for each }\ i=1,\ldots,p, \ \mbox{and satisfy}\ x(\tau_i)=x(\tau_i^-)\Big\}, 
	\end{align*}
	with the norm $\left\|x\right\|_{\mathrm{PC}(J;\mathbb{X})}:=\sup\limits_{t\in J}\left\|x(t)\right\|_{\mathbb{X}}$.
	\begin{lem}[Lemma 2.3, \cite{ER}]\label{lema2.1}
		Let $x:(-\infty,T]\rightarrow \mathbb{X}$ be a function such that $x_0=\psi$ and $x|_J\in \mathrm{PC}(J;\mathbb{X})$. Then
		$$\left\|x_s\right\|_\mathfrak{B}\le K_1\left\|\psi\right\|_\mathfrak{B}+K_2\sup\{\left\|x(\theta)\right\|_\mathbb{X}:\theta \in [0,\max\{0,s\}]\},\ s\in\mathfrak{R}(\varrho^-)\cup J,$$
		where $$K_1=\sup_{t \in \mathfrak{R}(\varrho-)}\varTheta^\psi(t)+\sup_{t\in J} \mathcal{Q}(t),\ K_2=\sup_{t\in J}\mathcal{P}(t). $$
	\end{lem}
	\begin{Ex}\label{exm2.8}
		Let $\mathfrak{B}=\mathcal{PC}_{g}(\mathbb{X})$ be the space consisting of all mappings $\phi:(-\infty,0]\to\mathbb{X}$ such that $\phi$ is left continuous, $\phi\vert_{[-q,0]}\in \mathcal{PC}([-q,0];\mathbb{X}),$ for $q>0,$ and $\int_{-\infty}^{0}\frac{\left\|\phi(\theta)\right\|_{\mathbb{X}}}{g(\theta)}\mathrm{d}\theta<\infty.$ The norm  $\left\|\cdot\right\|_{\mathfrak{B}}$ is defined as 
		\begin{align}
			\label{Bnor}\left\|\phi\right\|_{\mathfrak{B}}:=\int_{-\infty}^{0}\frac{\left\|\phi(\theta)\right\|_{\mathbb{X}}}{g(\theta)}\mathrm{d}\theta,
		\end{align} where the function $g:(-\infty,0]\to[1,\infty)$ satisfies the following conditions:
		\begin{enumerate}
			\item $g(\cdot)$ is continuous with $g(0)=1$, and $\lim\limits_{\theta\to-\infty} g(\theta)=\infty,$
			\item the function $\mathcal{O}(t):=\sup\limits_{-\infty<\theta\leq -t}\frac{g(t+\theta)}{g(\theta)}$ is locally bounded for $t\geq 0.$
		\end{enumerate}
	\end{Ex}
	\subsection{Mild solution, resolvent operator and assumptions}In order to define the mild solution of the system \eqref{1.1}, we first  consider Mainardi's Wright-type function $\mathrm{M}_{\upsilon}$ (cf. \cite{Fm},\cite{Ip})
	\begin{align*}
		\mathrm{M}_{\upsilon}(z)=\sum_{k=0}^{\infty}\frac{(-z)^{k}}{k!\Gamma(1-\upsilon(n+1))},\ \upsilon\in(0,1),\  z\in\mathbb{C}.
	\end{align*}
	Note that, for any $t>0$, $\mathrm{M}_{\upsilon}(t)\ge 0$ and it satisfies  the following: $$\int_{0}^{\infty} s^{c}\mathrm{M}_{\upsilon}(s)\mathrm{d}s=\frac{\Gamma(1+\delta)}{\Gamma(1+\upsilon\delta)},\ \mbox{ for }\ -1<c<\infty.$$ 
	For any $x\in\mathbb{X}$ and $\gamma=\frac{\alpha}{2}$, we define 
	\begin{align}
		\label{2.3}\mathcal{C}_{\gamma}(t)&:=\int_{0}^{\infty}\mathrm{M}_{\gamma}(\theta)\mathrm{C}(t^{\gamma}\theta)\mathrm{d}\theta,\\ \mathcal{T}_{\gamma}(t)&:=\int_{0}^{t}\mathcal{C}_{\gamma}(s)\mathrm{d}s, \\  
		\mathcal{S}_{\gamma}(t)&:=\int_{0}^{\infty} \gamma\theta\mathrm{M}_{\gamma}(\theta)\mathrm{S}(t^{\gamma}\theta)\mathrm{d}\theta.
	\end{align}
	Next, we discuss some important properties of the operators $\mathcal{C}_{\gamma}(t), \mathcal{T}_{\gamma}(t)$ and $\mathcal{S}_{\gamma}(t)$ for $t\ge0$ in the following lemma.
	\begin{lem}(\cite{Yz})\label{lem2.11}
		The operators $\mathcal{C}_{\gamma}(t), \mathcal{T}_{\gamma}(t)$ and $\mathcal{S}_{\gamma}(t)$ have the following properties:
		\begin{itemize}
			\item [(i)] The operators $\mathcal{C}_{\gamma}(t), \mathcal{T}_{\gamma}(t)$ and $\mathcal{S}_{\gamma}(t)$ are linear. Moreover, for all $t\geq 0$,
			\begin{align*}
				\left\|\mathcal{C}_{\gamma}(t)\right\|_{\mathcal{L}(\mathbb{X})}\le M, \ \ \left\|\mathcal{T}_{\gamma}(t)\right\|_{\mathcal{L}(\mathbb{X})}\le Mt,\ \ \left\|\mathcal{S}_{\gamma}(t)\right\|_{\mathcal{L}(\mathbb{X})}\le\frac{M}{\Gamma(2\gamma)}t^{\gamma}.
			\end{align*}
			\item [(ii)] For all $t\ge 0$, the operator $\mathcal{C}_{\gamma}(t)$ is strongly continuous and the operators $\mathcal{T}_{\gamma}(t)$, $\mathcal{S}_{\gamma}(t)$ and $t^{\gamma-1}\mathcal{S}_{\gamma}(t)$ are uniformly continuous.
			\item [(iii)] If $\mathrm{S}(t)$ is compact for $t\ge0$, then the operator $\mathcal{S}_{\gamma}(t)$ is also compact for $t\ge0$. 
		\end{itemize}
	\end{lem}
	\begin{Def}
		A function $x(\cdot;\psi,\eta,u):(-\infty, T]\to\mathbb{X}$ is called a  \emph{mild solution} of \eqref{1.1}, if $x(t)=\psi(t)$ on $t\in(-\infty,0]$ (where $\psi\in\mathfrak{B}$) and $x(\cdot)$ satisfies the following:
		\begin{equation}\label{2.2}
			x(t)=
			\begin{dcases}
				\mathcal{C}_{\gamma}(t)\psi(0)+\mathcal{T}_{\gamma}(t)\eta+\int_{0}^{t}(t-s)^{\gamma-1}\mathcal{S}_{\gamma}(t-s)\left[\mathrm{B}u(s)+f(s,x_{\varrho(s, x_s)})\right]\mathrm{d}s,\\ \qquad\qquad\qquad\qquad\qquad\qquad\qquad\qquad\qquad \ t\in[0, \tau_1],\\
				h_j(t, x(\tau_j^-)),\qquad\qquad\qquad\qquad\qquad\qquad\quad t\in(\tau_j, s_j],\ j=1,\ldots,p,\\
				\mathcal{C}_{\gamma}(t-s_j)h_j(s_j, x(\tau_j^-))+\mathcal{T}_{\gamma}(t-s_j)h'_j(s_j, x(\tau_j^-))\\ \quad -\int_{0}^{s_j}(s_j-s)^{\gamma-1}\mathcal{S}_{\gamma}(s_j-s)\left[\mathrm{B}u(s)+f(s,x_{\varrho(s, x_s)})\right]\mathrm{d}s \\\quad+\int_{0}^{t}(t-s)^{\gamma-1}\mathcal{S}_{\gamma}(t-s)\left[\mathrm{B}u(s)+f(s,x_{\varrho(s, x_s)})\right]\mathrm{d}s, \\\qquad\qquad\qquad\qquad\qquad\qquad\qquad\qquad\qquad \ t\in(s_j,\tau_{j+1}],\ j=1,\ldots,p.
			\end{dcases}
		\end{equation}
	\end{Def}
	\begin{Def}
		The system \eqref{1.1} is said to be approximately controllable on $J$, if for any initial function $\psi \in \mathfrak{B}, x_T\in\mathbb{X}$ and given $\epsilon>0$, there exist a control function $u\in\mathrm{L}^{2}(J;\mathbb{X})$ such that $$\left\|x(T)-x_T\right\|_{\mathbb{X}}\le\epsilon,$$ where $x(\cdot)$ is the mild solution of the system \eqref{1.1} with the control $u(\cdot)$.
	\end{Def}
	To investigate the approximate controllability of the system \eqref{1.1}, we first introduce the following operators:
	\begin{equation}\label{2.1}
		\left\{
		\begin{aligned}
			L_0^Tu&:=\int^{T}_{0}(T-t)^{\gamma-1}\mathcal{S}_{\gamma}(T-t)\mathrm{B}u(t)\mathrm{d}t,\\
			\Phi_{0}^{T}&:=\int^{T}_{0}(T-t)^{\gamma-1}\mathcal{S}_{\gamma}(T-t)\mathrm{B}\mathrm{B}^*\mathcal{S}_{\gamma}(T-t)^*\mathrm{d}t,\\
			\mathcal{R}(\lambda,\Phi_{0}^{T})&:=(\lambda \mathrm{I}+\Phi_{0}^{T}\mathscr{J})^{-1},\ \lambda > 0,
		\end{aligned}
		\right.
	\end{equation}
	where $\mathrm{B}^{*}$ and $\mathcal{S}_{\gamma}(t)^*$ represent the adjoint operators of $\mathrm{B}$ and $\mathcal{S}_{\gamma}(t),$ respectively. It is straightforward that the operators $L_0^T$ and $\Phi_{0}^{T}$ are linear and bounded. Moreover, the mapping $\mathscr{J} : \mathbb{X} \rightarrow 2^{\mathbb{X}^*}$ is defined as 
	\begin{align*}
		\mathscr{J}[x]&=\{x^* \in \mathbb{X}^* : \langle x, x^* \rangle=\left\|x\right\|_{\mathbb{X}}^2= \left\|x^*\right\|_{\mathbb{X}^*}^2 \}, \ \mbox{ for all }\  x\in \mathbb{X},
	\end{align*}
which	is known as the duality mapping. The notation  $\langle \cdot, \cdot  \rangle $ denotes the duality pairing between $\mathbb{X}$ and its dual $\mathbb{X}^*$. Since $\mathbb{X}$ is a reflexive Banach space, then $\mathbb{X}^*$ becomes strictly convex (see \cite{AA}), which ensures that the mapping $\mathscr{J}$ is strictly monotonic, bijective and demicontinuous, that is, $$x_k\to x\ \mbox{ in }\ \mathbb{X}\ \mbox{ implies }\ \mathscr{J}[x_k] \xrightharpoonup{w} \mathscr{J}[x]\ \mbox{ in } \ \mathbb{X}^*\ \mbox{ as }\ k\to\infty.$$ 
	\begin{rem}\label{rem2.8}
		If $\mathbb{X}$ is a separable Hilbert space, which is identified by its own dual, then the resolvent operator is defined as $\mathcal{R}(\lambda,\Phi_{0}^{T}):=(\lambda \mathrm{I}+\Phi_{0}^{T})^{-1},\ \lambda > 0$. 
	\end{rem}
	\begin{lem}[Lemma 2.2 \cite{M}]\label{lem2.9}
		For $\lambda>0$ and every $h\in\mathbb{X}$, the equation
		\begin{align}\label{2.4}\lambda z_{\lambda}+\Phi_{0}^{T}\mathscr{J}[z_{\lambda}]=\lambda h,\end{align}
		has a unique solution  which is given as 
		$$z_{\lambda}(h)=\lambda(\lambda \mathrm{I}+\Phi_{0}^{T}\mathscr{J})^{-1}(h)=\lambda\mathcal{R}(\lambda,\Phi_{0}^{T})(h)$$ and satisfies the following  \begin{align}\label{2.5}
			\left\|z_{\lambda}(h)\right\|_{\mathbb{X}}=\left\|\mathscr{J}[z_{\lambda}(h)]\right\|_{\mathbb{X}^*}\leq\left\|h\right\|_{\mathbb{X}}.
		\end{align}
		\begin{proof}
			Since the operator $\Phi_0^T$ is linear and bounded, then proceeding similar way as in the proof of Lemma 2.2 \cite{M}, one can complete the proof.
		\end{proof}
	\end{lem}
	We need to impose the following assumptions to obtain the approximate controllability of the system \eqref{1.1}.
	\begin{Ass}\label{as2.1} 
		\begin{enumerate}
			\item [\textit{($H0$)}] For every $h\in\mathbb{X}$, $z_\lambda=z_{\lambda}(h)=\lambda\mathcal{R}(\lambda,\Phi_{0}^{T})(h) \rightarrow 0$ as $\lambda\downarrow 0$ in strong topology, where $z_{\lambda}(h)$ is a solution of the equation \eqref{2.4}.
			\item [\textit{(H1)}] 
			\begin{enumerate} 
				\item [(i)] Let $x:(-\infty, T]\rightarrow \mathbb{X}$ be such that $x_0=\psi$ and $x|_{J}\in \mathrm{PC}(J;\mathbb{X}).$ The mapping $t\mapsto f(t, x_{\varrho(t,x_{t})}) $ is strongly measurable on $J$  and the function $f(t,\cdot): \mathfrak{B}\rightarrow \mathbb{X}$ is continuous for a.e. $t\in J$. Also, the map $t\mapsto f(s,x_{t})$ is continuous on $\mathfrak{R}(\varrho^-) \cup J,$ for every $s\in J$. 
				\item [(ii)] For each positive integer $r$, there exists a constant $\delta\in[0,\gamma]$ and a function $\phi_r\in \mathrm{L}^{\frac{1}{\delta}}(J;\mathbb{R^{+}})$, such that $$  \sup_{\left\| \psi\right\|_{\mathfrak{B}}\leq r} \left\|f(t, \psi)\right\|_{\mathbb{X}}\leq \phi_{r}(t), \mbox{ for a.e.} \ t \in J \mbox{ and } \psi\in \mathfrak{B},$$ with
				$$ \liminf_{r \rightarrow \infty } \frac {\left\|\phi_r\right\|_{\mathrm{L}^{\frac{1}{\delta}}(J;\mathbb{R^+})}}{r} = \zeta< \infty. $$ 
			\end{enumerate}
			\item [\textit{(H2)}] The impulses $ h_j:[\tau_j,s_j]\times\mathbb{X}\to\mathbb{X}$, for  $j=1,\dots,p$, are such that 
			\begin{itemize}
				\item [(i)] the impulses $h_j(\cdot,x):[\tau_j,s_j]\to\mathbb{X}$ are continuously differentiable for each $x\in\mathbb{X}$, 
				\item [(ii)] for all  $t\in[\tau_j,s_j]$, the impulses $h_j(t,\cdot):\mathbb{X}\to\mathbb{X}$ and their derivatives $h'_j(t,\cdot):\mathbb{X}\to\mathbb{X}$ are completely continuous,
				\item[(iii)] $\left\|h_{j}(t,x)\right\|_{\mathbb{X}}\le \kappa_{j},\ \left\|h'_{j}(t,x)\right\|_{\mathbb{X}}\le \vartheta_{j}, \text{for all}\ x\in\mathbb{X},\ t\in [\tau_{j},s_{j}],\ j=1,\cdots,p$, where $\kappa_{j}$'s and $\vartheta_{j}$'s are positive constants.
			\end{itemize} 
		\end{enumerate}
	\end{Ass}
	The following version of the discrete Gronwall-Bellman lemma (cf. \cite{Ch}) is used in the sequel.  
	\begin{lem}\label{lem2.13}
		If $\{f_n\}_{n=0}^{\infty}, \{g_n\}_{n=0}^{\infty}$ and $\{w_n\}_{n=0}^{\infty}$ are non-negative sequences and $$f_n\le g_n+\sum_{0\le k<n}w_kf_k,\ \text{ for }\ n\geq 0,$$  then $$f_n\le g_n+\sum_{0\le k<n}g_kw_k\exp\left(\sum_{k<j<n}g_j\right),\text{  for }\ n\geq 0.$$
	\end{lem}
	\section{Fractional order linear control problem}\label{sec3}\setcounter{equation}{0} This section deals with the linear fractional control problem and investigate the approximate controllability of the linear system. To achieve this goal, first we consider the linear-regulator problem with the cost functional defined as 
	\begin{equation}\label{3.1}
		\mathscr{G}(x,u)=\left\|x(T)-x_{T}\right\|^{2}_{\mathbb{X}}+\lambda\int^{T}_{0}(T-t)^{\gamma-1}\left\|u(t)\right\|^{2}_{\mathbb{U}}\mathrm{d}t,
	\end{equation}
	where $\gamma=\frac{\alpha}{2}$ with $\alpha\in(1,2)$, $\lambda >0$, $x_{T}\in \mathbb{X}$ and $x(\cdot)$ is the solution of the linear fractional control system:
	\begin{equation}\label{3.2}
		\left\{
		\begin{aligned}
			^CD_{0,t}^{\alpha}x(t)&= \mathrm{A}x(t)+\mathrm{B}u(t),\ t\in J,\\
			x(0)&=v,\ x'(0)=w,
		\end{aligned}
		\right.
	\end{equation}
	with the control $u(\cdot)\in \mathbb{U}$. Since $\mathrm{B}u\in\mathrm{L}^1(J;\mathbb{X})$ for any $u \in\mathrm{L}^{2}(J;\mathbb{U})=\mathcal{U}_{\mathrm{ad}}$ (admissible control class), then, $(t-s)^{\gamma-1}\mathcal{S}_{\gamma}(t-s)\mathrm{B}u(s)$ is integrable. Hence, the above linear system possess a unique mild solution $x\in\mathrm{C}(J;\mathbb{X})$ (see section 4.2, Chapter 4, \cite{P}) given by
	\begin{align}
		x(t)&=\mathcal{C}_{\gamma}(t)v+\mathcal{T}_{\gamma}(t)w+\int_{0}^{t}(t-s)^{\gamma-1}\mathcal{S}_{\gamma}(t-s)\mathrm{B}u(s)\mathrm{d}s,\ t\in J.
	\end{align}
	Next, we define the {\em admissible class} $\mathcal{A}_{\mathrm{ad}}$ for the optimal control problem \eqref{3.1}-\eqref{3.2} as $$\mathcal{A}_{\mathrm{ad}}=\{(x,u):x \text{ is the unique mild solution of } \eqref{3.2} \text{ with the control  }u\in \mathcal{U}_{\mathrm{ad}}\}.$$ For any $u \in\mathrm{L}^{2}(J;\mathbb{U}),$ the existence of a unique mild solution ensures  that the set $\mathcal{A}_{\mathrm{ad}}$ is nonempty.	Our next aim is to prove the existence of an optimal pair, which minimize the cost functional \eqref{3.1}.
	\begin{theorem}\label{th1} For given $v,w\in\mathbb{X}$, there exists a unique optimal pair  $(x^0,u^0)\in\mathcal{A}_{\mathrm{ad}}$ of the problem:
		\begin{align}\label{F}
			\min\limits_{(x,u)\in \mathcal{A}_{\mathrm{ad}}}\mathscr{G}(x, u).
		\end{align}
	\end{theorem} 
	\begin{proof}
		Let us assume $$L := \inf \limits _{u \in \mathcal{U}_{\mathrm{ad}}}\mathscr{G}(x,u).$$ Since $0\leq L< +\infty$, there exists a minimizing sequence $\{u^n\}_{n=1}^{\infty} \in \mathcal{U}_{\mathrm{ad}}$ such that $$\lim_{n\to\infty}\mathscr{G}(x^n,u^n) = L,$$ where $(x^n, u^n)\in\mathcal{A}_{\mathrm{ad}}$, for each $n\in\mathbb{N}$. Moreover, $x^n(\cdot)$ satisfies
		\begin{align}\label{3.4}
			x^n(t)&=\mathcal{C}_{\gamma}(t)v+\mathcal{T}_{\gamma}(t)w+\int_{0}^{t}(t-s)^{\gamma-1}\mathcal{S}_{\gamma}(t-s)\mathrm{B}u^n(s)\mathrm{d}s,\ t\in J.
		\end{align} 
		Since  $0\in\mathcal{U}_{\mathrm{ad}}$, without loss of generality, we may assume that $\mathscr{G}(x^n,u^n) \leq \mathscr{G}(x,0)$, where $(x,0)\in\mathscr{A}_{\mathrm{ad}}$. Using the definition of $\mathscr{G}(\cdot,\cdot)$, we  obtain
		\begin{align}\label{35}
			T^{\gamma-1}\lambda\int^{T}_{0}\left\|u^n(t)\right\|^{2}_{\mathbb{U}}\mathrm{d}t&\leq  \left\|x^n(T)-x_{T}\right\|^{2}_{\mathbb{X}}+\lambda\int^{T}_{0}(T-t)^{\gamma-1}\left\|u^n(t)\right\|^{2}_{\mathbb{U}}\mathrm{d}t\nonumber\\&\leq \left\|x(T)-x_{T}\right\|^{2}_{\mathbb{X}}\leq 2\left(\|x(T)\|_{\mathbb{X}}^2+\|x_T\|_{\mathbb{X}}^2\right)<+\infty.
		\end{align}
		The above estimate guarantees the existence of a large $R>0$ (independent of $n$) such that 
		\begin{align}\label{3.5}\int_0^T \|u^n(t)\|^2_{\mathbb{U}} \mathrm{d} t \leq R < +\infty .\end{align}
		Using the relation \eqref{3.4}, we estimate
		\begin{align*}
			\left\|x^n(t)\right\|_{\mathbb{X}}&\le\left\|\mathcal{C}_{\gamma}(t)v\right\|_{\mathbb{X}}+\left\|\mathcal{T}_{\gamma}(t)w\right\|_{\mathbb{X}}+\left\|\int_{0}^{t}(t-s)^{\gamma-1}\mathcal{S}_{\gamma}(t-s)\mathrm{B}u^n(s)\mathrm{d}s\right\|_{\mathbb{X}}\nonumber\\&\le M\left\|v\right\|_{\mathbb{X}}+Mt\left\|w\right\|_{\mathbb{X}}+\frac{M\tilde{M}}{\Gamma(2\gamma)}\int_{0}^{t}(t-s)^{2\gamma-1}\left\|u^n(s)\right\|_{\mathbb{U}}\mathrm{d}s
			\nonumber\\&\le M\left\|v\right\|_{\mathbb{X}}+MT\left\|w\right\|_{\mathbb{X}}+\frac{M\tilde{M}}{\Gamma(2\gamma)}\left(\frac{T^{4\gamma-1}}{4\gamma-1}\right)^{\frac{1}{2}}\left(\int_{0}^{t}\left\|u^n(s)\right\|_{\mathbb{U}}^2\mathrm{d}s\right)^{\frac{1}{2}}\nonumber\\&\le M\left\|v\right\|_{\mathbb{X}}+MT\left\|w\right\|_{\mathbb{X}}+\frac{M\tilde{M}}{\Gamma(2\gamma)}\left(\frac{T^{4\gamma-1}R}{4\gamma-1}\right)^{\frac{1}{2}}<+\infty,
		\end{align*}
		for all  $t\in J$. Since the space $\mathrm{L}^{2}(J;\mathbb{X})$ is reflexive, then the Banach-Alaoglu theorem yields the existence of a subsequence $\{x^{n_k}\}_{k=1}^{\infty}$ of $\{x^n\}_{n=1}^{\infty}$ such that 
		\begin{align}\label{3.6}
			x^{n_k}\xrightharpoonup{w}x^0\ \mbox{ in }\ \mathrm{L}^{2}(J;\mathbb{X}), \ \mbox{ as }\ k\to\infty. 
		\end{align}
		The estimate \eqref{3.5} implies that the sequence $\{u^n\}_{n=1}^{\infty}$ is uniformly bounded in the space $\mathrm{L}^2(J;\mathbb{U})$. Once again, by applying the Banach-Alaoglu theorem, we can find a subsequence, say,  $\{u^{n_k}\}_{k=1}^{\infty}$ of $\{u^n\}_{n=1}^{\infty}$ such that 
		\begin{align*}
			u^{n_k}\xrightharpoonup{w}u^0\ \mbox{ in }\ \mathrm{L}^2(J;\mathbb{U})=\mathcal{U}_{\mathrm{ad}}, \ \mbox{ as }\ k\to\infty. 
		\end{align*}
		Since the linear operator $\mathrm{B}$ is bounded (continuous) from $\mathbb{U}$ to $\mathbb{X}$, then we have 
		\begin{align}\label{3.7}
			\mathrm{B}	u^{n_k}\xrightharpoonup{w}\mathrm{B}u^0\ \mbox{ in }\ \mathrm{L}^2(J;\mathbb{X}),\ \mbox{ as }\ k\to\infty.
		\end{align}
		Moreover, by using the above  convergences together with the compactness of the operator $(\mathrm{Q}f)(\cdot) =\int_{0}^{\cdot}(\cdot-s)^{\gamma-1}\mathcal{S}_{\gamma}(\cdot-s)f(s)\mathrm{d}s:\mathrm{L}^2(J;\mathbb{X})\rightarrow \mathrm{C}(J;\mathbb{X}) $ (see Lemma \ref{lem2.12} below), we obtain
		\begin{align*}
			\left\|\int_{0}^{t}(t-s)^{\gamma-1}\mathcal{S}_{\gamma}(t-s)\mathrm{B}u^{n_k}(s)\mathrm{d}s-\int_{0}^{t}(t-s)^{\gamma-1}\mathcal{S}_{\gamma}(t-s)\mathrm{B}u(s)\mathrm{d}s\right\|_{\mathbb{X}}\to0,
		\end{align*}
		as $k\to\infty$,	for all $t\in J$. We now estimate 
		\begin{align}\label{3.8}
			&\left\|x^{n_k}(t)-x^*(t)\right\|_{\mathbb{X}}\nonumber\\&=\left\|\int_{0}^{t}(t-s)^{\gamma-1}\mathcal{S}_{\gamma}(t-s)\mathrm{B}u^{n_k}(s)\mathrm{d}s-\int_{0}^{t}(t-s)^{\gamma-1}\mathcal{S}_{\gamma}(t-s)\mathrm{B}u^{0}(s)\mathrm{d}s\right\|_{\mathbb{X}}\nonumber\\&\to 0\ \mbox{ as } \ k\to\infty, \ \mbox{for all }\ t\in J,
		\end{align}
		where 
		\begin{align*}
			x^*(t)=\mathcal{C}_{\gamma}(t)v+\mathcal{T}_{\gamma}(t)w+\int_{0}^{t}(t-s)^{\gamma-1}\mathcal{S}_{\gamma}(t-s)\mathrm{B}u^{0}(s)\mathrm{d}s,\ t\in J.
		\end{align*}
		The above expression guarantees that the function  $x^*\in \mathrm{C}(J;\mathbb{X})$ is the unique mild solution of the equation \eqref{3.2} with the control $u^{0}\in\mathcal{U}_{\mathrm{ad}}$. Using the convergences \eqref{3.6} and \eqref{3.8} and the fact  that the weak limit is unique, we deduce that $x^*(t)=x^0(t),$ for all $t\in J$. Therefore, the function $x^0$ is the unique mild solution of the system \eqref{3.2} with the control $u^{0}\in\mathcal{U}_{\mathrm{ad}}$, so that the whole sequence  $x^n\to x^0\in\mathrm{C}(J;\mathbb{X})$. Hence, we have  $(x^0,u^0)\in\mathscr{A}_{\mathrm{ad}}$.
		
		Next, we show that the functional $\mathscr{G}(\cdot,\cdot)$ attains its minimum at  $(x^0,u^0)$, that is, \emph{$\mathscr{G}=\mathscr{G}(x^0,u^0)$}. Note that the cost functional $\mathscr{G}(\cdot,\cdot)$ given in \eqref{3.1} is continuous and convex (see Proposition III.1.6 and III.1.10,  \cite{EI}) on $\mathrm{L}^2(J;\mathbb{X}) \times \mathrm{L}^2(J;\mathbb{U})$, it follows that $\mathscr{G}(\cdot,\cdot)$ is weakly lower semi-continuous (Proposition II.4.5, \cite{EI}). That is, for a sequence 
		$$(x^n,u^n)\xrightharpoonup{w}(x^0,u^0)\ \mbox{ in }\ \mathrm{L}^2(J;\mathbb{X}) \times  \mathrm{L}^2(J;\mathbb{U}),\ \mbox{ as }\ n\to\infty,$$
		we have 
		\begin{align*}
			\mathscr{G}(x^0,u^0) \leq  \liminf \limits _{n\rightarrow \infty} \mathscr{G}(x^n,u^n).
		\end{align*}
		Hence, we obtain 
		\begin{align*} L \leq \mathscr{G}(x^0,u^0) \leq  \liminf \limits _{n\rightarrow \infty} \mathscr{G}(x^n,u^n)= \lim \limits _{n\rightarrow \infty} \mathscr{G}(x^n,u^n) =L,\end{align*}
		and thus $(x^0,u^0)$ is a minimizer of the problem \eqref{F}. Since the cost functional given in \eqref{3.1} is convex, the constraint \eqref{3.2} is linear and the class $\mathcal{U}_{\mathrm{ad}}=\mathrm{L}^2(J;\mathbb{U})$ is convex, then the  optimal control obtained above is unique.
	\end{proof}
	
	\begin{rem}
		Since $\mathbb{X}$ is a separable reflexive Banach space with strictly convex dual $\mathbb{X}^*$, then the fact 8.12 in \cite{MFb} ensures that the norm $\|\cdot\|_{\mathbb{X}}$ is Gateaux differentiable. Moreover, the Gateaux derivative $\partial_x$ of the function $\varphi(x)=\frac{1}{2}\|x\|_{\mathbb{X}}^2$ is the duality map, that is, $$\langle\partial_x\varphi(x),y\rangle=\lim_{\epsilon \to 0}\frac{\varphi(x+\epsilon y)-\varphi(x)}{\epsilon}=\frac{1}{2}\frac{\mathrm{d}}{\mathrm{d}\epsilon}\|x+\epsilon y\|_{\mathbb{X}}^2\Big|_{\epsilon=0}=\langle\mathscr{J}[x],y\rangle,$$ for $y\in\mathbb{X}$. Furthermore, as we know that  $\mathbb{U}$ is a separable Hilbert space (identified with its own dual), then the space $\mathbb{U}$ admits a Fr\'echet differentiable norm (cf. Theorem 8.24, \cite{MFb}) .	
	\end{rem}
	The expression for the optimal control in the feedback form is obtained in the following lemma:
	
	\begin{lem}\label{lem3.1}
		Assume that $(x,u)$ is the optimal pair for the problem \eqref{F}. Then the optimal control $u$ is given by
		\begin{align*}
			u(t)=\mathrm{B}^*\mathcal{S}_{\gamma}(T-t)^*\mathscr{J}\left[\mathcal{R}(\lambda,\Phi_{0}^T)\ell(x(\cdot))\right],\  t\in [0, T),\ \lambda>0,\ \frac{1}{2}<\gamma<1, 
		\end{align*}
		with
		\begin{align*}
			\ell(x(\cdot))=x_{T}-\mathcal{C}_{\gamma}(T)v-\mathcal{T}_{\gamma}(T)w.
		\end{align*}
	\end{lem}
	\begin{proof}
		Let us consider the functional 
		\begin{align*}
			\mathscr{L}(\varepsilon)=\mathscr{G}(x_{u+\varepsilon z},u+\varepsilon z),
		\end{align*}
		where $(x,u)$ is the optimal solution of \eqref{F} and $z\in \mathrm{L}^{2}(J;\mathbb{U})$. Also the function $x_{u+\varepsilon z}$ is the unique mild solution of \eqref{3.2} with the control $u+\varepsilon z$. Then, it is immediate that 
		\begin{align*}
			x_{u+\varepsilon z}(t)= \mathcal{C}_{\gamma}(t)v+\mathcal{T}_{\gamma}(t)w+\int_{0}^{t}(t-s)^{\gamma-1}\mathcal{S}_{\gamma}(t-s)\mathrm{B}(u+\varepsilon z)(s)\mathrm{d}s.
		\end{align*}
		It is clear that the functional $\mathscr{L}(\varepsilon)$ has a critical point at $\varepsilon=0$. We now calculate the first variation of the cost functional $\mathscr{G}$ (defined in \eqref{3.1}) as
		\begin{align*}
			\frac{\mathrm{d}}{\mathrm{d}\varepsilon}\mathscr{L}(\varepsilon)\Big|_{\varepsilon=0}&=\frac{\mathrm{d}}{\mathrm{d}\varepsilon}\bigg[\left\|x_{u+\varepsilon z}(T)-x_{T}\right\|^{2}_{\mathbb{X}}+\lambda\int^{T}_{0}(T-t)^{\gamma-1}\left\|u(t)+\varepsilon z(t)\right\|^{2}_{\mathbb{U}}\mathrm{d}t\bigg]_{\varepsilon=0}\nonumber\\
			&=2\bigg[\langle \mathscr{J}(x_{u+\varepsilon z}(T)-x_{T}), \frac{\mathrm{d}}{\mathrm{d}\varepsilon}(x_{u+\varepsilon z}(T)-x_{T})\rangle\nonumber\\&\qquad +2\lambda\int^{T}_{0}(T-t)^{\gamma-1}(u(t)+\varepsilon z(t),\frac{\mathrm{d}}{\mathrm{d}\varepsilon}(u(t)+\varepsilon z(t)))\mathrm{d}t\bigg]_{\varepsilon=0}\nonumber\\
			&=2\left\langle\mathscr{J}(x(T)-x_T),\int_0^T(T-t)^{\gamma-1}\mathcal{S}_{\gamma}(T-t)\mathrm{B}z(t)\mathrm{d}t\right\rangle\\&\quad+2\lambda\int_0^T(T-t)^{\gamma-1}(u(t),z(t))\mathrm{d} t. 
		\end{align*}
		By taking the first variation zero, we deduce that
		\begin{align}\label{3.9}
			0&=\left\langle\mathscr{J}(x(T)-x_T),\int_0^T(T-t)^{\gamma-1}\mathcal{S}_{\gamma}(T-t)\mathrm{B}z(t)\mathrm{d}t\right\rangle+\lambda\int_0^T(T-t)^{\gamma-1}(u(t),z(t))\mathrm{d} t\nonumber\\&=\int_0^T(T-t)^{\gamma-1}\left\langle\mathscr{J}(x(T)-x_T),\mathcal{S}_{\gamma}(T-t)\mathrm{B}z(t) \right\rangle\mathrm{d}t +\lambda\int_0^T(T-t)^{\gamma-1}(u(t),z(t))\mathrm{d}t\nonumber\\&= \int_0^T(T-t)^{\gamma-1}\left(\mathrm{B}^*\mathcal{S}_{\gamma}(T-t)^*\mathscr{J}(x(T)-x_T)+\lambda u(t),z(t) \right)\mathrm{d}t,
		\end{align}
		where $(\cdot,\cdot)$ is the inner product in the Hilbert space $\mathbb{U}$. Note that the function $z\in \mathrm{L}^{2}(J;\mathbb{U})$ is an arbitrary (one can take $z$ to be $\mathrm{B}^*\mathcal{S}_{\gamma}(T-t)^*\mathscr{J}(x(T)-x_T)+\lambda u(t)$), it is immediate that
		\begin{align}\label{3.10}
			u(t)&= -\lambda^{-1}\mathrm{B}^*\mathcal{S}_{\gamma}(T-t)^*\mathscr{J}(x(T)-x_T),
		\end{align}
		for a.e. $t\in [0,T]$. Along with the relations \eqref{3.9} and \eqref{3.10}, it is clear that $u\in\mathrm{C}([0,T];\mathbb{U})$. Using the above expression of control, we find
		\begin{align}\label{3.11}
			x(T)&=\mathcal{C}_{\gamma}(T)v+\mathcal{T}_{\gamma}(T)w-\int^{T}_{0}\lambda^{-1}(T-s)^{\gamma-1}\mathcal{S}_{\gamma}(T-s)\mathrm{B}\mathrm{B}^*\mathcal{S}_{\gamma}(T-s)^*\mathscr{J}(x(T)-x_T)\mathrm{d}s\nonumber\\
			&=\mathcal{C}_{\gamma}(T)v+\mathcal{T}_{\gamma}(T)w-\lambda^{-1}\Phi_{0}^T\mathscr{J}\left[x(T)-x_{T}\right].
		\end{align}
		Let us assume
		\begin{align}\label{3.12}
			\ell(x(\cdot)):=x_{T}-\mathcal{C}_{\gamma}(T)v-\mathcal{T}_{\gamma}(T)w.
		\end{align}
		Further, by \eqref{3.11} and \eqref{3.12}, we have
		\begin{align}\label{3.13}
			x(T)-x_{T}&=-\ell(x(\cdot))-\lambda^{-1}\Phi_{0}^T\mathscr{J}\left[x(T)-x_{T}\right].
		\end{align}
		From \eqref{3.13}, one can easily get 
		\begin{align}\label{3.15}
			x(T)-x_T=-\lambda\mathrm{I}(\lambda\mathrm{I}+\Phi_0^T\mathscr{J})^{-1}\ell(x(\cdot))=-\lambda\mathcal{R}(\lambda,\Phi_0^T)\ell(x(\cdot)).
		\end{align}
		Finally, from \eqref{3.10}, we find the optimal control as
		\begin{align*}
			u(t)=\mathrm{B}^*\mathcal{S}_{\gamma}(T-t)^*\mathscr{J}\left[\mathcal{R}(\lambda,\Phi_{0}^T)\ell(x(\cdot))\right],\ \mbox{ for all } \ t\in [0,T],
		\end{align*}
		which completes the proof. 
	\end{proof}
	In the next lemma, we discuss the compactness of the operator $(\mathrm{Q}f)(\cdot) =\int_{0}^{\cdot}(\cdot-s)^{\gamma-1}\mathcal{S}_{\gamma}(\cdot-s) f(s)\mathrm{d}s:\mathrm{L}^2(J;\mathbb{X})\rightarrow \mathrm{C}(J;\mathbb{X}) ,\ \mbox{where}\ \gamma\in(\frac{1}{2},1)$ and $\mathbb{X}$ is a general Banach space.
	\begin{lem}\label{lem2.12}
		Suppose that the operator $\mathcal{S}_{\gamma}(t)$ is compact for $t\ge 0$. Let the operator $\mathrm{Q}:\mathrm{L}^{2}(J;\mathbb{X})\rightarrow \mathrm{C}(J;\mathbb{X})$ be defined as
		\begin{align}
			(\mathrm{Q}f)(t)= \int^{t}_{0}(t-s)^{\gamma-1}\mathcal{S}_{\gamma}(t-s)f(s)\mathrm{d}s, \ t\in J,\ \frac{1}{2}<\gamma<1.
		\end{align}
		Then the operator $\mathrm{Q}$ is compact.
	\end{lem}
	A proof of the above lemma is a straightforward adaptation of the proof of Lemma 3.2, \cite{SMJ}. 
	\begin{lem}\label{lem3.4}
		The linear control system \eqref{3.2} is approximately controllable on $J$ if and only if Assumption (H0) holds. 
	\end{lem}
	A proof of the above lemma can be obtained by proceeding similarly as in the proof of  Theorem 3.2, \cite{SM}.
	\begin{rem}\label{rem3.4}
		Note that the operator $\Phi_{0}^{T}$ is positive if and only if  Assumption (\textit{H0}) holds (see, Theorem 2.3, \cite{M}). The positivity of $\Phi_{0}^{T}$ is equivalent to $$ \langle x^*, \Phi_{0}^{T}x^*\rangle=0\Rightarrow x^*=0.$$ Using the definition of $\Phi_{0}^{T}$, we have 
		\begin{align}
			\langle x^*, \Phi_{0}^{T}x^*\rangle =\int_0^T(T-t)^{\gamma-1}\left\|\mathrm{B}^*\mathcal{S}_{\gamma}(T-t)^*x^*\right\|_{\mathbb{X}^*}^2\mathrm{d}t.
		\end{align}
		The above fact and Lemma \ref{lem3.4} guarantee that the approximate controllability of the linear system \eqref{3.2} is equivalent to the condition $$\mathrm{B}^*\mathcal{S}_{\gamma}(T-t)^*x^*=0,\ 0\le t<T \Rightarrow x^*=0.$$
	\end{rem}
	\section{Approximate Controllability of the Fractional order Impulsive System} \label{semilinear}\setcounter{equation}{0}
	The present section is reserved for the approximate controllability of the semilinear impulsive fractional control system \eqref{1.1}. For this, we first verify the existence of a mild solution of the  system \eqref{1.1} with the control defined as
	\begin{align}\label{C}
		u^{\gamma}_{\lambda}(t)&=\sum_{j=0}^{p}u^{\gamma}_{j,\lambda}(t)\chi_{[s_j, \tau_{j+1}]}(t), \ t\in J,\  \frac{1}{2}<\gamma<1,
	\end{align}
	where 
	\begin{align*}
		u^{\gamma}_{j,\lambda}(t)&=\mathrm{B}^*\mathcal{S}_{\gamma}(\tau_{j+1}-t)^*\mathscr{J}\left[\mathcal{R}(\lambda,\Phi_{s_j}^{\tau_{j+1}})g_j(x(\cdot))\right],
	\end{align*}
	for $t\in [s_j, \tau_{j+1}],j=0,1,\ldots,p$, with
	\begin{align*}
		g_0(x(\cdot))&=\xi_{0}-\mathcal{C}_{\gamma}(\tau_1)\psi(0)-\mathcal{T}_{\gamma}(\tau_1)\eta-\int^{\tau_1}_{0}(\tau_1-s)^{\gamma-1}\mathcal{S}_{\gamma}(\tau_1-s)f(s,\tilde{x}_{\varrho(s,\tilde{x}_s)})\mathrm{d}s,\nonumber\\
		g_j(x(\cdot))&=\xi_{j}-\mathcal{C}_{\gamma}(\tau_{j+1}-s_j)h_j(s_j,\tilde{x}(\tau_j^-))-\mathcal{T}_{\gamma}(\tau_{j+1}-s_j)h'_j(s_j,\tilde{x}(\tau_j^-))\\&\quad+\int_{0}^{s_j}(s_j-s)^{\gamma-1}\mathcal{S}_{\gamma}(s_j-s)\left[f(s,\tilde{x}_{\varrho(s,\tilde{x}_s)})+\mathrm{B}\sum_{k=0}^{j-1}u^{\gamma}_{k,\lambda}(s)\chi_{[s_k, \tau_{k+1})}(s)\right]\mathrm{d}s \\&\quad-\int_{0}^{\tau_{j+1}}(\tau_{j+1}-s)^{\gamma-1}\mathcal{S}_{\gamma}(\tau_{j+1}-s)f(s,\tilde{x}_{\varrho(s,\tilde{x}_s)})\mathrm{d}s	\\&\quad-\int_{0}^{s_j}(\tau_{j+1}-s)^{\gamma-1}\mathcal{S}_{\gamma}(\tau_{j+1}-s)\mathrm{B}\sum_{k=0}^{j-1}u^{\gamma}_{k,\lambda}(s)\chi_{[s_k, \tau_{k+1})}(s)\mathrm{d}s, \ j=1,\ldots,p,
	\end{align*} 
	  the function $\tilde{x}:(-\infty,T]\rightarrow\mathbb{X}$ is given by $$\tilde{x}(t)=\psi(t), \ t\in(-\infty,0], \ \tilde{x}(t)=x(t),\ t\in J,$$ and for arbitrary $\xi_{j}\in \mathbb{X}$ for $j=0,1,\ldots,p$. 
	\begin{rem}
		Since the operator $\Phi_{s_j}^{\tau_{j+1}},$  for each $j=0,\ldots,p,$ is linear, bounded and non-negative,   Lemma \ref{lem2.9} is valid for each $\Phi_{s_j}^{\tau_{j+1}},$  for $j=0,\ldots,p$.
	\end{rem}
	In the following theorem, we obtain the existence of a mild solution of the system \eqref{1.1} with the control given in \eqref{C}.
	\begin{theorem}\label{thm4.3}
		If Assumptions (R1)-(R3) and (H1)-(H2) hold true. Then for fixed $\xi_{j}\in\mathbb{X},$ for $j=0,1,\ldots,p$ and every $\lambda>0$, the system \eqref{1.1} with the control \eqref{C} has at least one mild solution on $J$, provided 
		\begin{align}\label{cnd}
			\frac{2MT^{2\gamma-\delta}K_{2}\zeta}{\Gamma(2\gamma)c^{1-\delta}}\left\{1+\frac{(p+1)(p+2)R}{2}+\frac{p(p+1)R^2}{2}\sum_{k=0}^{p-1}e^{\frac{(p+k)(p-k-1)R}{2}}\right\}<1,
		\end{align}
		where $R=\frac{2T^{3\gamma}}{3\gamma\lambda}\left(\frac{M\tilde{M}}{\Gamma(2\gamma)}\right)^2$ and $c=\frac{2\gamma-\delta}{1-\delta}$.
	\end{theorem}
	\begin{proof}
		Let us take a set $\mathcal{Z}:=\{x\in\mathrm{PC}(J;\mathbb{X}) : x(0)=\psi(0)\}$ with the norm $\left\|\cdot\right\|_{\mathrm{PC}(J;\mathbb{X})}$. Next, we consider a set $\mathcal{E}_{r}=\{x\in\mathcal{Z} : \left\|x\right\|_{\mathrm{PC}(J;\mathbb{X})}\le r\}$, for each $r>0$.
		
		For $\lambda>0$, we define an operator $\mathscr{F}_{\lambda}:\mathcal{Z}\to\mathcal{Z}$ such that
		\begin{equation}\label{2}
			(\mathscr{F}_{\lambda}x)(t)=\left\{
			\begin{aligned}
				&\mathcal{C}_{\gamma}(t)\psi(0)+\mathcal{T}_{\gamma}(t)\eta+\int_{0}^{t}(t-s)^{\gamma-1}\mathcal{S}_{\gamma}(t-s)\left[\mathrm{B}u^{\gamma}_{\lambda}(s)+f(s,\tilde{x}_{\varrho(s, \tilde{x}_s)})\right]\mathrm{d}s,\\ &\qquad\qquad\qquad\qquad\qquad\qquad\qquad\qquad\qquad t\in[0, \tau_1],\\
				&h_j(t, \tilde{x}(\tau_j^-)),\qquad\qquad\qquad\qquad\qquad\qquad\ \  t\in(\tau_j, s_j], j=1,\ldots,p,\\
				&\mathcal{C}_{\gamma}(t-s_j)h_j(s_j, \tilde{x}(\tau_j^-))+\mathcal{T}_{\gamma}(t-s_j)h'_j(s_j, \tilde{x}(\tau_j^-))\\& \quad -\int_{0}^{s_j}(s_j-s)^{\gamma-1}\mathcal{S}_{\gamma}(s_j-s)\left[\mathrm{B}u^{\gamma}_{\lambda}(s)+f(s,\tilde{x}_{\varrho(s, \tilde{x}_s)})\right]\mathrm{d}s \\&\quad+\int_{0}^{t}(t-s)^{\gamma-1}\mathcal{S}_{\gamma}(t-s)\left[\mathrm{B}u^{\gamma}_{\lambda}(s)+f(s,\tilde{x}_{\varrho(s, \tilde{x}_s)})\right]\mathrm{d}s,\\ &\qquad\qquad\qquad\qquad\qquad\qquad\qquad\qquad\qquad t\in(s_j,  \tau_{j+1}],\ j=1,\ldots,p.
			\end{aligned}
			\right.
		\end{equation}
		where $u^{\gamma}_{\lambda}(\cdot)$ is given in \eqref{C}. It is clear form the definition of $ \mathscr{F}_{\lambda}$ that the system $\eqref{1.1}$ has a mild solution, if the operator $ \mathscr{F}_{\lambda}$ has a fixed point. The proof that the operator $\mathscr{F}_{\lambda}$ has a fixed point is divided into the following steps. 
		\vskip 0.1in 
		\noindent\textbf{Step (1): } \emph{$ F_{\lambda}(\mathcal{E}_r)\subset \mathcal{E}_r,$ for some $ r $}. We achieve this goal by contradiction. To prove this, we assume that  for any $\lambda>0$ and every $r > 0$, there exists $x^r(\cdot) \in \mathcal{E}_r$ such that $\left\|(\mathscr{F}_\lambda x^r)(t)\right\|_{\mathbb{X}} > r$, for some $t\in J$ ($t$ may depend upon $r$). First,  we compute
		\begin{align}\label{4.4}
			\left\|g_0(x(\cdot))\right\|_{\mathbb{X}}&\le\left\|\xi_0\right\|_{\mathbb{X}}+\left\|\mathcal{C}_{\gamma}(\tau_1)\psi(0)\right\|_{\mathbb{X}}+\left\|\mathcal{T}_{\gamma}(\tau_1)\eta\right\|_{\mathbb{X}}\nonumber\\&\quad+\int_{0}^{\tau_1}(\tau_1-s)^{\gamma-1}\left\|\mathcal{S}_{\gamma}(\tau_1-s)f(s,\tilde{x}_{\varrho(s,\tilde{x}_s)})\right\|_{\mathbb{X}}\mathrm{d}s\nonumber\\&\le \left\|\xi_0\right\|_{\mathbb{X}}+M\left\|\psi(0)\right\|_{\mathbb{X}}+Mt\left\|\eta\right\|_{\mathbb{X}}+\frac{M}{\Gamma(2\gamma)}\int_{0}^{\tau_1}(\tau_1-s)^{2\gamma-1}\phi_ {r'}(s)\mathrm{d}s\nonumber\\&\le\left\|\xi_0\right\|_{\mathbb{X}}+M\left\|\psi(0)\right\|_{\mathbb{X}}+MT\left\|\eta\right\|_{\mathbb{X}}\nonumber\\&\quad+\frac{M}{\Gamma(2\gamma)}\left(\int_{0}^{\tau}(\tau_1-s)^{\frac{2\gamma-1}{1-\delta}}\mathrm{d}s\right)^{1-\delta}\left(\int_{0}^{\tau_1}(\phi_{r'}(s))^{\frac{1}{\delta}}\mathrm{d}s\right)^{\delta}\nonumber\\&\le\left\|\xi_0\right\|_{\mathbb{X}}+M\left\|\psi(0)\right\|_{\mathbb{X}}+MT\left\|\eta\right\|_{\mathbb{X}}+\frac{M\tau_{1}^{2\gamma-\delta}}{\Gamma(2\gamma)c^{1-\delta}}\left\|\phi_{r'}\right\|_{\mathrm{L}^{\frac{1}{\delta}}([0,\tau_1];\mathbb{R}^+)}\nonumber\\&\le\left\|\xi_0\right\|_{\mathbb{X}}+M\left\|\psi(0)\right\|_{\mathbb{X}}+MT\left\|\eta\right\|_{\mathbb{X}}+\frac{2MT^{2\gamma-\delta}}{\Gamma(2\gamma)c^{1-\delta}}\left\|\phi_{r'}\right\|_{\mathrm{L}^{\frac{1}{\delta}}(J;\mathbb{R}^+)}=:N_0,
		\end{align}  
		where $c=\frac{2\gamma-\delta}{1-\delta}$ and $r'=K_{1}\left\|\psi\right\|_{\mathfrak{B}}+K_{2}r$. Next, we estimate
		\begin{align*}
			\left\|g_j(x(\cdot))\right\|_{\mathbb{X}}&\le\left\|\xi_j\right\|_{\mathbb{X}}+\left\|\mathcal{C}_{\gamma}(t-s_j)h_j(s_j,\tilde{x}(\tau_j^-))\right\|_{\mathbb{X}}+\left\|\mathcal{T}_{\gamma}(t-s_j)h'_j(s_j,\tilde{x}(t_j^-))\right\|_{\mathbb{X}}\nonumber\\&\quad+\int_{0}^{s_j}(s_j-s)^{\gamma-1}\left\|\mathcal{S}_{\gamma}(s_j-s)f(s,\tilde{x}_{\varrho(s,\tilde{x}_s)})\right\|_{\mathbb{X}}\mathrm{d}s\nonumber\\&\quad+\int_{0}^{\tau_{j+1}}(\tau_{j+1}-s)^{\gamma-1}\left\|\mathcal{S}_{\gamma}(\tau_{j+1}-s)f(s,\tilde{x}_{\varrho(s,\tilde{x}_s)})\right\|_{\mathbb{X}}\mathrm{d}s\nonumber\\&\quad+\int_{0}^{s_j}(s_j-s)^{\gamma-1}\left\|\mathcal{S}_{\gamma}(s_j-s)\mathrm{B}\sum_{k=0}^{j-1}u^\gamma_{k,\lambda}(s)\chi_{[s_k,\tau_{k+1}]}(s)\right\|_{\mathbb{X}}\mathrm{d}s\nonumber\\&\quad+\int_{0}^{s_j}(\tau_{j+1}-s)^{\gamma-1}\left\|\mathcal{S}_{\gamma}(\tau_{j+1}-s)\mathrm{B}\sum_{k=0}^{j-1}u^\gamma_{k,\lambda}(s)\chi_{[s_k,\tau_{k+1}]}(s)\right\|_{\mathbb{X}}\mathrm{d}s\nonumber\\&\le\left\|\xi_j\right\|_{\mathbb{X}}+M\kappa_{j}+M(t-s_j)\vartheta_{j}+\frac{M}{\Gamma(2\gamma)}\int_{0}^{s_j}(s_j-s)^{2\gamma-1}\phi_{r'}(s)\mathrm{d}s\nonumber\\&\quad+\frac{M}{\Gamma(2\gamma)}\int_{0}^{\tau_{j+1}}(\tau_{j+1}-s)^{2\gamma-1}\phi_{r'}(s)\mathrm{d}s+\frac{1}{\lambda}\left(\frac{M\tilde{M}}{\Gamma(2\gamma)}\right)^2\sum_{k=0}^{j-1}\left\|g_k(x(\cdot))\right\|_{\mathbb{X}}\nonumber\\&\qquad\times\int_{s_k}^{\tau_{k+1}}\left[(s_j-s)^{2\gamma-1}+(\tau_{j+1}-s)^{2\gamma-1}\right](\tau_{k+1}-s)^{\gamma}\mathrm{d}s\nonumber\\&\le\left\|\xi_j\right\|_{\mathbb{X}}+M\kappa_{j}+MT\vartheta_{j}+\frac{2M}{\Gamma(2\gamma)}\int_{0}^{\tau_{j+1}}(\tau_{j+1}-s)^{2\gamma-1}\phi_{r'}(s)\mathrm{d}s\nonumber\\&\quad+\frac{2}{\lambda}\left(\frac{M\tilde{M}}{\Gamma(2\gamma)}\right)^2\sum_{k=0}^{j-1}\left\|g_k(x(\cdot))\right\|_{\mathbb{X}}\int_{s_k}^{\tau_{k+1}}(\tau_{j+1}-s)^{2\gamma-1}(\tau_{k+1}-s)^{\gamma}\mathrm{d}s\nonumber\\&\le\left\|\xi_j\right\|_{\mathbb{X}}+M\kappa_{j}+MT\vartheta_{j}+\frac{2M}{\Gamma(2\gamma)}\int_{0}^{\tau_{j+1}}(\tau_{j+1}-s)^{2\gamma-1}\phi_{r'}(s)\mathrm{d}s\nonumber\\&\quad+\frac{2}{\lambda}\left(\frac{M\tilde{M}}{\Gamma(2\gamma)}\right)^2\sum_{k=0}^{j-1}\left\|g_k(x(\cdot))\right\|_{\mathbb{X}}\int_{s_k}^{\tau_{k+1}}(\tau_{j+1}-s)^{3\gamma-1}\mathrm{d}s\nonumber\\&\le\left\|\xi_j\right\|_{\mathbb{X}}+M\kappa_{j}+MT\vartheta_{j}+\frac{2M\tau_{j+1}^{2\gamma-\delta}}{\Gamma(2\gamma)c^{1-\delta}}\left\|\phi_{r'}\right\|_{\mathrm{L}([0,\tau_{j+1}];\mathbb{R}^+)}\nonumber\\&\quad+\frac{2}{\lambda}\left(\frac{M\tilde{M}}{\Gamma(2\gamma)}\right)^2\sum_{k=0}^{j-1}\left\|g_k(x(\cdot))\right\|_{\mathbb{X}}\frac{(\tau_{j+1}-s_k)^{3\gamma}-(\tau_{j+1}-\tau_{k+1})^{3\gamma}}{3\gamma}\nonumber\\&\le\left\|\xi_j\right\|_{\mathbb{X}}+M\kappa_{j}+MT\vartheta_{j}+\frac{2MT^{2\gamma-\delta}}{\Gamma(2\gamma)c^{1-\delta}}\left\|\phi_{r'}\right\|_{\mathrm{L}(J;\mathbb{R}^+)}\nonumber\\&\quad+\frac{2T^{3\gamma}}{3\gamma\lambda}\left(\frac{M\tilde{M}}{\Gamma(2\gamma)}\right)^2\sum_{k=0}^{j-1}\left\|g_k(x(\cdot))\right\|_{\mathbb{X}}\nonumber\\&\le N_j+R\sum_{k=0}^{j-1}\left\|g_k(x(\cdot))\right\|_{\mathbb{X}},
		\end{align*}
		where  $R=\frac{2T^{3\gamma}}{3\gamma\lambda}\left(\frac{M\tilde{M}}{\Gamma(2\gamma)}\right)^2$ and $N_j=\left\|\xi_j\right\|_{\mathbb{X}}+M\kappa_{j}+M T\vartheta_{j} +\frac{2M T^{2\gamma-\delta}}{\Gamma(2\gamma)c^{1-\delta}}\left\|\phi_{r'}\right\|_{\mathrm{L}(J;\mathbb{R}^+)},$ for $j=1,\ldots,p.$ Using the discrete Gronwall-Bellman lemma (Lemma \ref{lem2.13}), we obtain
		\begin{align}\label{4.5}
			\left\|g_j(x(\cdot))\right\|_{\mathbb{X}}\le&N_j+R\sum_{k=0}^{j-1}N_ke^{\frac{(j+k)(j-k-1)R}{2}}=:C_j, \ \mbox{ for }\ j=1,\ldots,p.
		\end{align}
		Taking $t\in[0,\tau_1]$ and using the estimates \eqref{2.5}, \eqref{4.4},\eqref{4.5}, Lemma \ref{lem2.11} and Assumption \ref{as2.1} \textit{(H1)}-\textit{(H2)}, we evaluate
		\begin{align}\label{4.21}
			r&<\left\|(\mathscr{F}_{\lambda}x^r)(t)\right\|_\mathbb{X}\nonumber\\&=\left\|\mathcal{C}_{\gamma}(t)\psi(0)+\mathcal{T}_{\gamma}(t)\eta+\int_{0}^{t}(t-s)^{\gamma-1}\mathcal{S}_{\gamma}(t-s)\left[\mathrm{B}u^{\gamma}_{\lambda}(s)+f(s,\tilde{x}_{\varrho(s, \tilde{x}_s)})\right]\mathrm{d}s\right\|_{\mathbb{X}}\nonumber\\&\le  M\left\|\psi(0)\right\|_{\mathbb{X}}+Mt\left\|\eta\right\|_{\mathbb{X}}+\frac{M\tilde{M}}{\Gamma(2\gamma)}\int_{0}^{t}(t-s)^{2\gamma-1}\left\|u^{\gamma}_{\lambda}(s)\right\|_{\mathbb{U}}\mathrm{d}s\nonumber\\&\quad+\frac{M}{\Gamma(2\gamma)}\int_{0}^{t}(t-s)^{2\gamma-1}\left\|f(s,\tilde{x}_{\rho(s,\tilde{x}_s)})\right\|_{\mathbb{X}}\mathrm{d}s\nonumber\\&\le M\left\|\psi(0)\right\|_{\mathbb{X}}+M\tau_{1}\left\|\eta\right\|_{\mathbb{X}}+\frac{M\tilde{M}}{\Gamma(2\gamma)}\int_{0}^{\tau_1}(\tau_1-s)^{2\gamma-1}\left\|u^{\gamma}_{\lambda}(s)\right\|_{\mathbb{U}}\mathrm{d}s\nonumber\\&\quad+\frac{M}{\Gamma(2\gamma)}\int_{0}^{\tau_1}(\tau_1-s)^{2\gamma-1}\left\|f(s,\tilde{x}_{\rho(s,\tilde{x}_s)})\right\|_{\mathbb{X}}\mathrm{d}s\nonumber\\&\le M\left\|\psi(0)\right\|_{\mathbb{X}}+M\tau_{1}\left\|\eta\right\|_{\mathbb{X}}+\frac{M\tilde{M}}{\Gamma(2\gamma)}\int_{0}^{\tau_1}(\tau_1-s)^{2\gamma-1}\left\|u^{\gamma}_{1,\lambda}(s)\right\|_{\mathbb{U}}\mathrm{d}s\nonumber\\&\quad+\frac{M}{\Gamma(2\gamma)}\int_{0}^{\tau_1}(\tau_1-s)^{2\gamma-1}\phi_{r'}(s)\mathrm{d}s\nonumber\\&\le M\left\|\psi(0)\right\|_{\mathbb{X}}+M\tau_1\left\|\eta\right\|_{\mathbb{X}}+\frac{1}{\lambda}\left(\frac{M\tilde{M}}{\Gamma(2\gamma)}\right)^{2}\left\|g_0((\cdot))\right\|_{\mathbb{X}}\int_{0}^{\tau_1}(\tau_1-s)^{3\gamma-1}\mathrm{d}s\nonumber\\&\quad+\frac{M}{\Gamma(2\gamma)}\int_{0}^{\tau_1}(\tau_1-s)^{2\gamma-1}\phi_{r'}(s)\mathrm{d}s\nonumber\\&\le M\left\|\psi(0)\right\|_{\mathbb{X}}+M\tau_1\left\|\eta\right\|_{\mathbb{X}}+\frac{N_{0}}{\lambda}\left(\frac{M\tilde{M}} {\Gamma(2\gamma)}\right) ^{2} \int_{0}^{\tau_1}(\tau_1-s)^{3\gamma-1}\mathrm{d}s\nonumber\\&\quad+\frac{M}{\Gamma(2\gamma)}\left(\int_{0}^{\tau_1}(\tau_1-s)^{\frac{2\gamma-1}{1-\delta}}\right)^{1-\delta}\left(\int_{0}^{\tau_1}\phi_{r'}^{\frac{1}{\delta}}(s)\mathrm{d}s\right)^{\delta}\nonumber\\&= M\left\|\psi(0)\right\|_{\mathbb{X}}+M\tau_1\left\|\eta\right\|_{\mathbb{X}}+\frac{N_{0}\tau_1^{3\gamma}}{3\gamma\lambda}\left(\frac{M\tilde{M}}{\Gamma(2\gamma)}\right)^{2}+ \frac {M\tau_1^{2\gamma-\delta}} {\Gamma(2\gamma)c^{1-\delta}}\left\|\phi_{r'}\right\|_{\mathrm{L}^{\frac{1}{\delta}}([0,\tau_1];\mathbb{R}^+)}\nonumber\\&\le M\left\|\psi(0)\right\|_{\mathbb{X}}+MT\left\|\eta\right\|_{\mathbb{X}}+\frac{N_{0}T^{3\gamma}}{3\gamma\lambda}\left(\frac{M\tilde{M}} {\Gamma(2\gamma)}\right) ^{2}+ \frac {MT^{2\gamma-\delta}} {\Gamma(2\gamma)c^{1-\delta}}\left\|\phi_{r'}\right\|_{\mathrm{L}^{\frac{1}{\delta}}(J;\mathbb{R}^+)}\nonumber\\&<M\left\|\psi(0)\right\|_{\mathbb{X}}+MT \left\|\eta\right\|_{\mathbb{X}}+\frac{2T^{3\gamma}}{3\gamma\lambda}\left(\frac{M\tilde{M}}{\Gamma(2\gamma)}\right)^{2}\sum_{k=0}^{p}C_k+\frac {2MT^{2\gamma-\delta}} {\Gamma(2\gamma)c^{1-\delta}}\left\|\phi_{r'}\right\|_{\mathrm{L}^{\frac{1}{\delta}}(J;\mathbb{R}^+)}\nonumber\\&=M\left\|\psi(0)\right\|_{\mathbb{X}}+MT\left\|\eta\right\|_{\mathbb{X}}+R\sum_{k=0}^{p}C_k+\frac{2MT^{2\gamma-\delta}}{\Gamma(2\gamma)c^{1-\delta}}\left\|\phi_{r'}\right\|_{\mathrm{L}^{\frac{1}{\delta}}(J;\mathbb{R}^+)},
		\end{align}
		where $C_0=N_0$. For $t\in(\tau_j,s_j],\ j=1,\ldots,p$, we obtain  
		\begin{align}\label{4.22}
			r&<\left\|(\mathscr{F}_{\lambda}x^r)(t)\right\|_\mathbb{X}\le\left\|h_j(t, \tilde{x}(\tau_j^-))\right\|_{\mathbb{X}}\nonumber\\&\le \kappa_{j}< \kappa_{j}+R\sum_{k=0}^{p}C_k+\frac{2MT^{2\gamma-\delta}}{\Gamma(2\gamma)c^{1-\delta}}\left\|\phi_{r'}\right\|_{\mathrm{L}^{\frac{1}{\delta}}}(J;\mathbb{R}^+).
		\end{align}
		Taking $t\in(s_j,\tau_{j+1}], \ j=1,\dots,p$, we compute
		\begin{align}\label{4.23}
			r&<\left\|(\mathscr{F}_{\lambda}x^r)(t)\right\|_\mathbb{X}\nonumber\\&=\bigg\|\mathcal{C}_{\gamma}(t-s_j)h_j(s_j,\tilde{x}(\tau_j^-))+\mathcal{T}_{\gamma}(t-s_j)h'_j(s_j,\tilde{x}(\tau_j^-))\nonumber\\&\quad-\int_{0}^{s_j}(s_j-s)^{\gamma-1}\mathcal{S}_{\gamma}(s_j-s)\left[\mathrm{B}u^{\gamma}_{\lambda}(s)+f(s,\tilde{x}_{\varrho(s,\tilde{x}_s)})\right]\mathrm{d}s\nonumber\\&\quad+\int_{0}^{t}(t-s)^{\gamma-1}\mathcal{S}_{\gamma}(t-s)\left[\mathrm{B}u^{\gamma}_{\lambda}(s)+f(s,\tilde{x}_{\varrho(s,\tilde{x}_s)})\right]\mathrm{d}s\bigg\|_{\mathbb{X}}\nonumber\\&\leq\left\|\mathcal{C}_{\gamma}(t-s_j)h_j(s_j,\tilde{x}(\tau_j^-))\right\|_{\mathbb{X}}+\left\|\mathcal{T}_{\gamma}(t-s_j)h'_j(s_j,\tilde{x}(t_j^-))\right\|_{\mathbb{X}}\nonumber\\&\quad+\int_{0}^{s_j}(s_j-s)^{\gamma-1}\left\|\mathcal{S}_{\gamma}(s_j-s)\mathrm{B}u^{\gamma}_{\lambda}(s)\right\|_{\mathbb{X}}\mathrm{d}s\nonumber\\&\quad+\int_{0}^{s_j}(s_j-s)^{\gamma-1}\left\|\mathcal{S}_{\gamma}(s_j-s)f(s,\tilde{x}_{\varrho(s,\tilde{x}_s)})\right\|_{\mathbb{X}}\mathrm{d}s\nonumber\\&\quad+\int_{0}^t(t-s)^{\gamma-1}\left\|\mathcal{S}_{\gamma}(t-s)\mathrm{B}u^{\gamma}_{\lambda}(s)\right\|_{\mathbb{X}}\mathrm{d}s\nonumber\\&\quad+\int_{0}^t(t-s)^{\gamma-1}\left\|\mathcal{S}_{\gamma}(t-s)f(s,\tilde{x}_{\varrho(s,\tilde{x}_s)})\right\|_{\mathbb{X}}\mathrm{d}s\nonumber\\&\le M\kappa_{j}+M(t-s_j)\vartheta_{j}+\frac{M\tilde{M}}{\Gamma(2\gamma)}\int_{0}^{s_j}(s_j-s)^{2\gamma-1}\left\|u_{\lambda}^{\gamma}( s)\right\|_{\mathbb{U}}\mathrm{d}s\nonumber\\&\quad+\frac{M}{\Gamma(2\gamma)}\int_{0}^{s_j}(s_j-s)^{2\gamma-1}\left\|f(s,\tilde{x}_{\varrho(s,\tilde{x}_s)})\right\|_{\mathbb{X}}+\frac{M\tilde{M}}{\Gamma(2\gamma)}\int_{0}^{t}(t-s)^{2\gamma-1}\left\|u_{\lambda}^{\gamma}(s)\right\|_{\mathbb{U}}\mathrm{d}s\nonumber\\&\quad+\frac{M}{\Gamma(2\gamma)}\int_{0}^{t}(t-s)^{2\gamma-1}\left\|f(s,\tilde{x}_{\varrho(s,\tilde{x}_s)})\right\|_{\mathbb{X}}\mathrm{d}s \nonumber\\&\leq M\kappa_{j}+MT\vartheta_{j}+\frac{M\tilde{M}}{\Gamma(2\gamma)}\int_{0}^{s_j}(s_j-s)^{2\gamma-1}\left\|\sum_{k=0}^{j-1}u_{k,\lambda}^{\gamma}(s)\chi_{[s_k,\tau_{k+1}]}(s)\right \|_{\mathbb{U}}\mathrm{d}s\nonumber\\&\quad+\frac{M}{\Gamma(2\gamma)}\int_{0}^{s_j}(s_j-s)^{2\gamma-1}\phi_{r'}(s)\mathrm{d}s\nonumber\\&\quad+\frac{M\tilde{M}}{\Gamma(2\gamma)}\int_{0}^{\tau_{j+1}}(\tau_{j+1}-s)^{2\gamma-1}\left\|\sum_{k=0}^{j}u_{k,\lambda}^{\gamma}(s)\chi_{[s_k,\tau_{k+1}]}(s)\right\|_{\mathbb{U}}\mathrm{d}s\nonumber\\&\quad+\frac{M}{\Gamma(2\gamma)}\int_{0}^{t_{j+1}}(t_{j+1}-s)^{2\gamma-1}\phi_{r'}(s)\mathrm{d}s\nonumber\\&\leq M\kappa_{j}+MT\vartheta_{j}+\frac{1}{\lambda}\left(\frac{M\tilde{M}}{\Gamma(2\gamma)}\right)^{2}\sum_{k=0}^{j-1}\left\|g_{k}(x(\cdot))\right\|_{\mathbb{X}}\int_{s_k}^{\tau_{k+1}}(s_j-s)^{2\gamma-1}(\tau_{k+1}-s)^{\gamma}\mathrm{d}s\nonumber\\&\quad+\frac{Ms_j^{2\gamma-\delta}}{\Gamma(2\gamma)c^{1-\delta}}\left\|\phi_{r'}\right\|_{\mathrm{L}^{\frac{1}{\delta}}([0,s_j];\mathbb{R}^+)}+\frac{Mt_{j+1}^{2\gamma-\delta}}{\Gamma(2\gamma)c^{1-\delta}}\left\|\phi_{r'}\right\|_{\mathrm{L}^{\frac{1}{\delta}}([0,\tau_{j+1}];\mathbb{R}^+)}\nonumber\\&\quad+\frac{1}{\lambda}\left(\frac{M\tilde{M}}{\Gamma(2\gamma)}\right)^{2}\sum_{k=0}^{j}\left\|g_{k}(x(\cdot))\right\|_{\mathbb{X}}\int_{s_k}^{\tau_{k+1}}(s_j-s)^{2\gamma-1}(\tau_{k+1}-s)^{\gamma}\mathrm{d}s\nonumber\\&\le M\kappa_{j}+MT\vartheta_{j}+\frac{T^{3\gamma}}{3\gamma\lambda}\left(\frac{M\tilde{M}}{\Gamma(2\gamma)}\right)^{2}\sum_{k=0}^{j-1}\left\|g_k(x(\cdot))\right\|_{\mathbb{X}}+\frac{2MT^{2\gamma-\delta}}{\Gamma(2\gamma)c^{1-\delta}}\left\|\phi_{r'}\right\|_{\mathrm{L}^{\frac{1}{\delta}}(J;\mathbb{R}^+)}\nonumber\\&\quad+\frac{T^{3\gamma}}{3\gamma\lambda}\left(\frac{M\tilde{M}}{\Gamma(2\gamma)}\right)^{2}\sum_{k=0}^{j}\left\|g_k(x(\cdot))\right\|_{\mathbb{X}}
			\nonumber\\&\leq M\kappa_{j}+MT\vartheta_{j}+\frac{2T^{3\gamma}}{3\gamma\lambda}\left(\frac{M\tilde{M}}{\Gamma(2\gamma)}\right)^{2}\sum_{k=0}^{j}\left\|g_k(x(\cdot))\right\|_{\mathbb{X}}+\frac{2MT^{2\gamma-\delta}}{\Gamma(2\gamma)c^{1-\delta}}\left\|\phi_{r'}\right\|_{\mathrm{L}^{\frac{1}{\delta}}(J;\mathbb{R}^+)}\nonumber\\&\le M\kappa_{j}+MT\vartheta_{j}+\frac{2T^{3\gamma}}{3\gamma\lambda}\left(\frac{M\tilde{M}}{\Gamma(2\gamma)}\right)^{2}\sum_{k=0}^{p}\left\|g_k(x(\cdot))\right\|_{\mathbb{X}}+\frac{2MT^{2\gamma-\delta}}{\Gamma(2\gamma)c^{1-\delta}}\left\|\phi_{r'}\right\|_{\mathrm{L}^{\frac{1}{\delta}}(J;\mathbb{R}^+)}\nonumber\\&\le M\kappa_{j}+MT\vartheta_{j}+\frac{2T^{3\gamma}}{3\gamma\lambda}\left(\frac{M\tilde{M}}{\Gamma(2\gamma)}\right)^{2}\sum_{k=0}^{p}C_k+\frac{2MT^{2\gamma-\delta}}{\Gamma(2\gamma)c^{1-\delta}}\left\|\phi_{r'}\right\|_{\mathrm{L}^{\frac{1}{\delta}}(J;\mathbb{R}^+)}\nonumber\\&=M\kappa_{j}+MT\vartheta_{j}+R\sum_{k=0}^{p}C_k+\frac{2MT^{2\gamma-\delta}}{\Gamma(2\gamma)c^{1-\delta}}\left\|\phi_{r'}\right\|_{\mathrm{L}^{\frac{1}{\delta}}(J;\mathbb{R}^+)}.
		\end{align}
		Using Assumption \ref{as2.1} (\textit{H2})(ii), we easily obtain
		\begin{align*}
			\liminf_{r \rightarrow \infty }\frac {\left\|\phi_{r'}\right\|_{\mathrm{L}^{\frac{1}{\delta}}(J;\mathbb{R^+})}}{r}&=\liminf_{r\rightarrow\infty}\left (\frac {\left\|\gamma_{r'}\right\|_{\mathrm{L}^{\frac{1}{\alpha_1}}(J;\mathbb{R^+})}}{r'}\times\frac{r'}{r}\right)=K_2\zeta.
		\end{align*}
		Thus, dividing by $r$ in the expressions \eqref{4.21}, \eqref{4.22}, \eqref{4.23} and then passing $r\to\infty$, we obtain
		\begin{align*}
			\frac{2MT^{2\gamma-\delta}K_{2}\zeta}{\Gamma(2\gamma)c^{1-\delta}}\left\{1+\frac{(p+1)(p+2)R}{2}+\frac{p(p+1)R^2}{2}\sum_{k=0}^{p-1}e^{\frac{(p+k)(p-k-1)R}{2}}\right\}>1,
		\end{align*}
		which is  a contradiction to \eqref{cnd}. Therefore, for some  $r>0$, $\mathscr{F}_{\lambda}(\mathcal{E}_{r})\subset \mathcal{E}_{r}.$
		\vskip 0.1in 
		\noindent\textbf{Step (2): } Next, we verify that the operator \emph{$ \mathscr{F}_{\lambda}$ is compact.} To prove this claim, we use the infinite-dimensional version of the Ascoli-Arzela theorem (see, Theorem 3.7, Chapter 2, \cite{JYONG}). According to this theorem, it is enough to verify the following:
		\begin{itemize}
			\item [(i)] the image of $\mathcal{E}_r$ under $\mathscr{F}_{\lambda}$ is uniformly bounded (which is proved in Step (1)),
			\item [(ii)] the image of $\mathcal{E}_r$ under $\mathscr{F}_{\lambda}$ is equicontinuous,
			\item [(iii)] for any $t\in J$, the set $\mathcal{V}(t)=\{(\mathscr{F}_\lambda x)(t):x\in \mathcal{E}_r\}$ is relatively compact.
		\end{itemize}
		First, we show that the map $\mathscr{F}_{\lambda}$ is equicontinuous, that is, the image of $\mathcal{E}_r$ under $\mathscr{F}_{\lambda}$ is equicontinuous. For $t_1,t_2\in[0,\tau_1]$ such that $t_1<t_2$ and $x\in\mathcal{E}_r$, we evaluate the following:
		\begin{align}\label{4.31}
			&\left\|(\mathscr{F}_{\lambda}x)(t_2)-(\mathscr{F}_{\lambda}x)(t_1)\right\|_{\mathbb{X}}\nonumber\\&\le\left\|\left[\mathcal{C}_{\gamma}(t_2)-\mathcal{C}_{\gamma}(t_1)\right]\psi(0)\right\|_{\mathbb{X}}+\left\|\left[\mathcal{T}_{\gamma}(t_2)-\mathcal{T}_{\gamma}(t_1)\right]\eta\right\|_{\mathbb{X}}\nonumber\\&\quad+\left\|\int_{t_1}^{t_2}(t_2-s)^{\gamma-1}\mathcal{S}_{\gamma}(t_2-s)f(s,\tilde{x}_{\varrho(s,\tilde{x}_s)})\mathrm{d}s\right\|_{\mathbb{X}}\nonumber\\&\quad+\left\|\int_{t_1}^{t_2}(t_2-s)^{\gamma-1}\mathcal{S}_{\gamma}(t_2-s)\mathrm{B}u^{\gamma}_{\lambda}(s)\mathrm{d}s\right\|_{\mathbb{X}}\nonumber\\&\quad+\left\|\int_{0}^{t_1}\left[(t_2-s)^{\gamma-1}-(t_1-s)^{\gamma-1}\right]\mathcal{S}_{\gamma}(t_2-s)\mathrm{B}u^{\gamma}_{\lambda}(s)\mathrm{d}s\right\|_{\mathbb{X}}\nonumber\\&\quad+\left\|\int_{0}^{t_1}(t_1-s)^{\gamma-1}\left[\mathcal{S}_{\gamma}(t_2-s)-\mathcal{S}_{\gamma}(t_1-s)\right]\mathrm{B}u^{\gamma}_{\lambda}(s)\mathrm{d}s\right\|_{\mathbb{X}}\nonumber\\&\quad+\left\|\int_{0}^{t_1}\left[(t_2-s)^{\gamma-1}-(t_1-s)^{\gamma-1}\right]\mathcal{S}_{\gamma}(t_2-s)f(s,\tilde{x}_{\varrho(s,\tilde{x}_s)})\mathrm{d}s\right\|_{\mathbb{X}}\nonumber\\&\quad+\left\|\int_{0}^{t_1}(t_1-s)^{\gamma-1}\left[\mathcal{S}_{\gamma}(t_2-s)-\mathcal{S}_{\gamma}(t_1-s)\right]f(s,\tilde{x}_{\varrho(s,\tilde{x}_s)})\mathrm{d}s\right\|_{\mathbb{X}}\nonumber\\&\leq\left\|\mathcal{C}_{\gamma}(t_2)\psi(0)-\mathcal{C}_{\gamma}(t_1)\psi(0)\right\|_{\mathbb{X}}+\left\|\mathcal{T}_{\gamma}(t_2)-\mathcal{T}_{\gamma}(t_1)\right\|_{\mathcal{L}(\mathbb{X})}\left\|\eta\right\|_{\mathbb{X}}\nonumber\\&\quad+\frac{M}{\Gamma(2\gamma)}\int_{t_1}^{t_2}(t_2-s)^{2\gamma-1}\phi_{r'}(s)\mathrm{d}s+\frac{N_0}{\lambda}\left(\frac{M\tilde{M}}{\Gamma(2\gamma)}\right)^{2}\int_{t_1}^{t_2}(t_2-s)^{2\gamma-1}(\tau_1-s)^{\gamma}\mathrm{d}s\nonumber\\&\quad+\frac{N_0}{\lambda}\left(\frac{M\tilde{M}}{\Gamma(2\gamma)}\right)^{2}\int_{0}^{t_1}\left|(t_2-s)^{\gamma-1}-(t_1-s)^{\gamma-1}\right|(t_2-s)^{\gamma-1}(\tau_1-s)^{\gamma}\mathrm{d}s\nonumber\\&\quad+\frac{M\tilde{M}^{2}N_{0}}{\lambda\Gamma(2\gamma)}\int_{0}^{t_1}(t_1-s)^{\gamma-1}\left\|\mathcal{S}_{\gamma}(t_2-s)-\mathcal{S}_{\gamma}(t_1-s)\right\|_{\mathcal{L}(\mathbb{X})}(\tau_1-s)^{\gamma}\mathrm{d}s\nonumber\\&\quad+\frac{M}{\Gamma(2\gamma)}\int_{0}^{t_1}\left|(t_2-s)^{\gamma-1}-(t_1-s)^{\gamma-1}\right|(t_2-s)^{\gamma}\phi_{r'}(s)\mathrm{d}s\nonumber\\&\quad+\int_{0}^{t_1}(t_1-s)^{\gamma-1}\left\|\mathcal{S}_{\gamma}(t_2-s)-\mathcal{S}_{\gamma}(t_1-s)\right\|_{\mathcal{L}(\mathbb{X})}\phi_{r'}(s)\mathrm{d}s\nonumber\\&\leq\left\|\mathcal{C}_{\gamma}(t_2)\psi(0)-\mathcal{C}_{\gamma}(t_1)\psi(0)\right\|_{\mathbb{X}}+\left\|\mathcal{T}_{\gamma}(t_2)-\mathcal{T}_{\gamma}(t_1)\right\|_{\mathcal{L}(\mathbb{X})}\left\|\eta\right\|_{\mathbb{X}}\nonumber\\&\quad+\frac{M}{\Gamma(2\gamma)}\frac{(t_2-t_1)^{2\gamma-\delta}}{c^{1-\delta}}\left(\int_{t_1}^{t_2}(\phi_{r'}(s))^{\frac{1}{\delta}}\mathrm{d}s\right)^{\delta}\nonumber\\&\quad+\frac{N_{0}}{\lambda}\left(\frac{M\tilde{M}}{\Gamma(2\gamma)}\right)^{2}\int_{t_1}^{t_2}(t_2-s)^{2\gamma-1}(\tau_1-s)^{\gamma}\mathrm{d}s\nonumber\\&\quad+\frac{N_{0}}{\lambda}\left(\frac{M\tilde{M}}{\Gamma(2\gamma)}\right)^{2}\int_{0}^{t_1}\left|(t_2-s)^{\gamma-1}-(t_1-s)^{\gamma-1}\right|(t_2-s)^{\gamma-1}(\tau_1-s)^{\gamma}\mathrm{d}s\nonumber\\&\quad+\frac{M\tilde{M}^{2}N_{0}}{\lambda\Gamma(2\gamma)}\sup_{s\in[0,t_1]}\left\|\mathcal{S}_{\gamma}(t_2-s)-\mathcal{S}_{\gamma}(t_1-s)\right\|_{\mathcal{L}(\mathbb{X})}\int_{0}^{t_1}(t_1-s)^{\gamma-1}(\tau_1-s)^{\gamma}\mathrm{d}s\nonumber\\&\quad+\frac{M}{\Gamma(2\gamma)}\int_{0}^{t_1}\left|(t_2-s)^{\gamma-1}-(t_1-s)^{\gamma-1}\right|(t_2-s)^{\gamma}\phi_{r'}(s)\mathrm{d}s\nonumber\\&\quad+\sup_{s\in[0,t_1]}\left\|\mathcal{S}_{\gamma}(t_2-s)-\mathcal{S}_{\gamma}(t_1-s)\right\|_{\mathcal{L}(\mathbb{X})}\int_{0}^{t_1}(t_1-s)^{\gamma-1}\phi_{r'}(s)\mathrm{d}s.
		\end{align}	
		Similarly for $t_1,t_2\in(s_j,\tau_{j+1}], \ j=1,\ldots,p$ with $t_1<t_2$ and $x\in\mathcal{E}_r$, one can compute
		\begin{align}\label{4.32}
			&\left\|(\mathscr{F}_{\lambda}x)(t_2)-(\mathscr{F}_{\lambda}x)(t_1)\right\|_{\mathbb{X}}\nonumber\\&\le\left\|\mathcal{C}_{\gamma}(t_2-s_j)h(s_j,x(\tau_j^-))-\mathcal{C}_{\gamma}(t_1-s_j)h(s_j,x(\tau_j^-))\right\|_{\mathbb{X}}\nonumber\\&\quad+\left\|\mathcal{T}_{\gamma}(t_2-s_j)-\mathcal{T}_{\gamma}(t_1-s_j)\right\|_{\mathcal{L}(\mathbb{X})}\vartheta_{j}+\frac{M}{\Gamma(2\gamma)}\frac{(t_2-t_1)^{2\gamma-\delta}}{c^{1-\delta}}\left(\int_{t_1}^{t_2}(\phi_{r'}(s))^{\frac{1}{\delta}}\mathrm{d}s\right)^{\delta}\nonumber\\&\quad+\frac{C_j}{\lambda}\left(\frac{M\tilde{M}}{\Gamma(2\gamma)}\right)^{2}\int_{t_1}^{t_2}(t_2-s)^{2\gamma-1}(\tau_{j+1}-s)^{\gamma}\mathrm{d}s\nonumber\\&\quad+\frac{M\tilde{M}}{\Gamma(2\gamma)}\int_{0}^{t_1}\left|(t_2-s)^{\gamma-1}-(t_1-s)^{\gamma-1}\right|(t_2-s)^{\gamma}\left\|u_{\lambda}^{\gamma}(s)\right\|_{\mathbb{U}}\mathrm{d}s\nonumber\\&\quad+\tilde{M}\sup_{s\in[0,t_1]}\left\|\mathcal{S}_{\gamma}(t_2-s)-\mathcal{S}_{\gamma}(t_1-s)\right\|_{\mathcal{L}(\mathbb{X})}\int_{0}^{t_1}(t_1-s)^{\gamma-1}\left\|u_{\lambda}^{\gamma}(s)\right\|_{\mathbb{U}}\mathrm{d}s\nonumber\\&\quad+\frac{M}{\Gamma(2\gamma)}\int_{0}^{t_1}\left|(t_2-s)^{\gamma-1}-(t_1-s)^{\gamma-1}\right|(t_2-s)^{\gamma}\phi_{r'}(s)\mathrm{d}s\nonumber\\&\quad+\sup_{s\in[0,t_1]}\left\|\mathcal{S}_{\gamma}(t_2-s)-\mathcal{S}_{\gamma}(t_1-s)\right\|_{\mathcal{L}(\mathbb{X})}\int_{0}^{t_1}(t_1-s)^{\gamma-1}\phi_{r'}(s)\mathrm{d}s\nonumber\\&\leq\left\|\mathcal{C}_{\gamma}(t_2-s_j)h(s_j,x(\tau_j^-))-\mathcal{C}_{\gamma}(t_1-s_j)h(s_j,x(\tau_j^-))\right\|_{\mathbb{X}}\nonumber\\&\quad+\left\|\mathcal{T}_{\gamma}(t_2-s_j)-\mathcal{T}_{\gamma}(t_1-s_j)\right\|_{\mathcal{L}(\mathbb{X})}\vartheta_{j}+\frac{M}{\Gamma(2\gamma)}\frac{(t_2-t_1)^{2\gamma-\delta}}{c^{1-\delta}}\left(\int_{t_1}^{t_2}(\phi_{r'}(s))^{\frac{1}{\delta}}\mathrm{d}s\right)^{\delta}\nonumber\\&\quad+\frac{C_j}{\lambda}\left(\frac{M\tilde{M}}{\Gamma(2\gamma)}\right)^{2}\int_{t_1}^{t_2}(t_2-s)^{2\gamma-1}(\tau_{j+1}-s)^{\gamma}\mathrm{d}s\nonumber\\&\quad+\sum_{k=0}^{j-1}\frac{C_k}{\lambda}\left(\frac{M\tilde{M}}{\Gamma(2\gamma)}\right)^{2}\int_{s_k}^{\tau_{k+1}}\left|(t_2-s)^{\gamma-1}-(t_1-s)^{\gamma-1}\right|(t_2-s)^{\gamma}(\tau_{k+1}-s)^{\gamma}\mathrm{d}s\nonumber\\&\quad+\frac{C_j}{\lambda}\left(\frac{M\tilde{M}}{\Gamma(2\gamma)}\right)^{2}\int_{s_j}^{s_1}\left|(t_2-s)^{\gamma-1}-t_1-s)^{\gamma-1}\right|(t_2-s)^{\gamma}(\tau_{j+1}-s)^{\gamma}\mathrm{d}s\nonumber\\&\quad+\sup_{s\in[0,s_1]}\left\|\mathcal{S}_{\gamma}(t_2-s)-\mathcal{S}_{\gamma}(t_1-s)\right\|_{\mathcal{L}(\mathbb{X})}\sum_{k=0}^{j-1}\frac{C_k}{\lambda}\frac{M\tilde{M}^{2}}{\Gamma(2\gamma)}\int_{s_k}^{\tau_{k+1}}(t_1-s)^{\gamma-1}(\tau_{k+1}-s)^{\gamma}\mathrm{d}s\nonumber\\&\quad+\sup_{s\in[0,s_1]}\left\|\mathcal{S}_{\gamma}(t_2-s)-\mathcal{S}_{\gamma}(t_1-s)\right\|_{\mathcal{L}(\mathbb{X})}\frac{C_j}{\lambda}\frac{M\tilde{M}^{2}}{\Gamma(2\gamma)}\int_{s_j}^{t_1}(t_1-s)^{\gamma-1}(t_{j+1}-s)^{\gamma}\mathrm{d}s\nonumber\\&\quad+\frac{M}{\Gamma(2\gamma)}\int_{0}^{t_1}\left|(t_2-s)^{\gamma-1}-(t_1-s)^{\gamma-1}\right|(t_2-s)^{\gamma}\phi_{r'}(s)\mathrm{d}s\nonumber\\&\quad+\sup_{s\in[0,t_1]}\left\|\mathcal{S}_{\gamma}(t_2-s)-\mathcal{S}_{\gamma}(t_1-s)\right\|_{\mathcal{L}(\mathbb{X})}\int_{0}^{t_1}(t_1-s)^{\gamma-1}\phi_{r'}(s)\mathrm{d}s.
		\end{align}
		Moreover, for $t_1,t_2\in(\tau_j, s_j], \ j=1,\ldots,p$ with $t_1<t_2$ and $x\in\mathcal{E}_r$, we have 
		\begin{align}\label{4.33}
			\left\|(\mathscr{F}_{\lambda}x)(t_2)-(\mathscr{F}_{\lambda}x)(t_1)\right\|_{\mathbb{X}}&\le\left\|h_j(t_2,\tilde{x}(\tau_j^-))-h_j(t_1,\tilde{x}(\tau_j^-))\right\|_{\mathbb{X}}.
		\end{align}
		From the  facts that the operator $\mathcal{C}_{\gamma}(t)x,\ x\in\mathbb{X}$ is uniformly continuous for $t\in J$, the operators $\mathcal{T}_{\gamma}(t)$ and $\mathcal{S}_{\gamma}(t)$ are uniformly continuous for all $t\in J$ and  the impulses $h_{j}(\cdot, x)$ are continuous  for $j=1,\ldots,p$ and each $x\in\mathbb{X}$, we conclude that the right hand side of the expressions \eqref{4.31}, \eqref{4.32} and \eqref{4.33} converge to zero as $|t_2-t_1|\rightarrow 0$. Hence, the image of $\mathcal{E}_r$ under $\mathscr{F}_\lambda$ is equicontinuous.
		
		The set $\mathcal{V}(t)=\{(\mathscr{F}_\lambda x)(t):x\in \mathcal{E}_r\},$ for each $t\in J$, is relatively compact in $\mathcal{E}_r$ follows from the facts that the  operator $\mathcal{S}_{\gamma}(t)$ is compact for $t\ge 0$, the impulses $h_{j}(t,\cdot),\ h_{j}'(t,\cdot)$, for $t\in[\tau_j,s_j], j=1,\ldots,p,$ are completely continuous and also the  operator $(\mathrm{Q}f)(\cdot) =\int_{0}^{\cdot}(\cdot-s)^{\gamma-1}\mathcal{S}_{\gamma}(\cdot-s)f(s)\mathrm{d}s$ is compact (see Lemma \ref{lem2.12}). Hence, for each $t\in J$, the set $\mathcal{V} (t) = \{(\mathscr{F}_{\lambda}x)(t) : x\in\mathcal{E}_r\},$ is relatively compact. 
		
		Thus, by applying the Arzela-Ascoli theorem, we obtain that the operator $\mathscr{F}_{\lambda}$  is compact.
		\vskip 0.1in 
		\noindent\textbf{Step (3): } In the final step, we prove that the operator \emph{$ \mathscr{F}_{\lambda}$ is continuous}. In order to prove this, we consider a sequence $\{{x}^n\}^\infty_{n=1}\subseteq \mathcal{E}_r$ such that ${x}^n\rightarrow {x}\mbox{ in }{\mathcal{E}_r},$ that is,
		$$\lim\limits_{n\rightarrow \infty}\left\|x^n-x\right\|_{\mathrm{PC}(J;\mathbb{X})}=0.$$
		From Lemma \ref{lema2.1}, we have
		\begin{align*}
			\left\|\tilde{x_{s}^n}- \tilde{x_{s}}\right\|_{\mathfrak{B}}&\leq K_{2}\sup\limits_{\theta\in J}\left\|\tilde{x^{n}}(\theta)-\tilde{x}(\theta)\right\|_{\mathbb{X}}=K_{2}\left\|x^{n}-x\right\|_{\mathrm{PC}(J;\mathbb{X})}\rightarrow 0 \ \mbox{ as } \ n\rightarrow\infty,
		\end{align*}
		for all $s\in\mathfrak{R}(\varrho^{-})\cup J$. Since $\varrho(s, \tilde{x_s^k})\in\mathfrak{R}(\varrho^{-})\cup J,$ for all $k\in\mathbb{N}$, then we conclude that
		\begin{align*}
			\left\|\tilde{x^n}_{\varrho(s,\tilde{x_s^k})}-\tilde{x}_{\varrho(s,\tilde{x_s^k})}\right\|_{\mathfrak{B}}\rightarrow 0 \ \mbox{ as } \ n\rightarrow\infty, \ \mbox{ for all }\ s\in J\  \mbox{ and }\ k\in\mathbb{N}.
		\end{align*}
		In particular, we choose $k=n$ and use the above convergence together with Assumption \ref{as2.1} \textit{(H1)} to  obtain
		\begin{align}\label{4.25}
			&	\left\|f(s,\tilde{x^n}_{\varrho(s,\tilde{x_s^n})})-f(s,\tilde{x}_{\varrho(s,\tilde{x_s})})\right\|_{\mathbb{X}}\nonumber\\&\leq\left\|f(s,\tilde{x^n}_{\varrho(s,\tilde{x_s^n})})-f(s,\tilde{x}_{\varrho(s,\tilde{x_s^n})})\right\|_{\mathbb{X}}+\left\|f(s, \tilde{x}_{\varrho(s,\tilde{x_s^n})})-f(s,\tilde{x}_{\varrho(s,\tilde{x_s})})\right\|_{\mathbb{X}}\nonumber\\&\to 0\ \mbox{ as }\ n\to\infty, \mbox{ uniformly for } \ s\in J. 
		\end{align}
		Using the convergence \eqref{4.25} and the dominated convergence theorem (DCT), we conclude
		\begin{align}\label{4.26}
			&\left\|g_0(x^n(\cdot))-g_0(x(\cdot))\right\|_{\mathbb{X}}\nonumber\\&\le\left\|\int^{\tau_1}_{0}(\tau_1-s)^{\gamma-1}\mathcal{S}_{\gamma}(\tau_1-s)\left[f(s,\tilde{x^n}_{\varrho(s,\tilde{x^n_s})})-f(s,\tilde{x}_{\varrho(s,\tilde{x_s})})\right]\mathrm{d}s\right\|_{\mathbb{X}}\nonumber\\&\le\frac{M}{\Gamma(2\gamma)}\int^{\tau_1}_{0}(\tau_1-s)^{2\gamma-1}\left\|f(s,\tilde{x^n}_{\varrho(s,\tilde{x^n_s})})-f(s,\tilde{x}_{\varrho(s,\tilde{x_s})})\right\|_{\mathbb{X}}\mathrm{d}s\nonumber\\&\to 0\ \mbox{ as }\ n\to\infty.
		\end{align}
		Invoking the above convergences and the relation \eqref{2.5}, we estimate
		\begin{align*}
			&	\left\|\mathcal{R}(\lambda,\Phi_{0}^{\tau_{1}})g_0(x^{n}(\cdot))-\mathcal{R}(\lambda,\Phi_{0}^{\tau_{1}})g_0(x(\cdot))\right\|_{\mathbb{X}}\nonumber\\&=\frac{1}{\lambda}\left\|\lambda\mathcal{R}(\lambda,\Phi_{0}^{\tau_1})\left(g_0(x^{n}(\cdot))-g_0(x(\cdot))\right)\right\|_{\mathbb{X}}\nonumber\\&\leq\frac{1}{\lambda}\left\|g_0(x^{n}(\cdot))-g_0(x(\cdot))\right\|_{\mathbb{X}}\nonumber\\&\to 0 \ \mbox{ as } \ n \to \infty.
		\end{align*}
		By using the demicontinuous property of the duality mapping $\mathscr{J}:\mathbb{X}\to\mathbb{X}^{*}$, one can easily obtain
		\begin{align}\label{4.2}
			\mathscr{J}\left[\mathcal{R}(\lambda,\Phi_{0}^{\tau_1})g_0(x^{n}(\cdot))\right]\xrightharpoonup{w}\mathscr{J}\left[\mathcal{R}(\lambda,\Phi_{0}^{\tau_1})g_0(x(\cdot))\right] \ \mbox{ as } \ n\to\infty  \ \mbox{ in }\ \mathbb{X}^{*}.
		\end{align}
		From Lemma \ref{lem2.11}, we infer that the operator $\mathcal{S}_{\gamma}(t)$ is compact for $t>0$. Consequently, the operator $\mathcal{S}_{\gamma}(t)^*$ is also compact for $t>0$. Thus, by using this compactness  and the weak convergence \eqref{4.2}, one can easily conclude that 
		\begin{align}\label{4.1}
			&	\left\|u^{n,\gamma}_{0,\lambda}(t)-u_{0,\lambda}^{\gamma}(t)\right||_{\mathbb{U}}\nonumber\\&\le\left\|\mathrm{B}^*\mathcal{S}_{\gamma}(\tau_{1}-t)^*\left(\mathscr{J}\left[\mathcal{R}(\lambda,\Phi_{0}^{\tau_{1}})g_0(x^{n}(\cdot))\right]-\mathscr{J}\left[\mathcal{R}(\lambda,\Phi_{0}^{\tau_{1}})g_0(x(\cdot))\right]\right)\right\|_{\mathbb{U}}\nonumber\\&\le\tilde{M}\left\|\mathcal{S}_{\gamma}(\tau_{1}-t)^*\left(\mathscr{J}\left[\mathcal{R}(\lambda,\Phi_{0}^{\tau_{1}})g_0(x^{n}(\cdot))\right]-\mathscr{J}\left[\mathcal{R}(\lambda,\Phi_{0}^{\tau_{1}})g_0(x(\cdot))\right]\right)\right\|_{\mathbb{X}}\nonumber\\&\to 0\ \mbox{ as }\ n\to\infty, \ \mbox{uniformly for}\ t\in[0,\tau_{1}].
		\end{align}
		Similarly, for $j=1$, we evaluate
		\begin{align}\label{4.27}
			&\left\|g_1(x^n(\cdot))-g_1(x(\cdot))\right\|_{\mathbb{X}}\nonumber\\&\le\left\|\mathcal{C}_{\gamma}(\tau_{2}-s_1)\left[h_1(s_1,\tilde{x^n}(\tau_1^-))-h_1(s_1,\tilde{x}(\tau_1^-))\right]\right\|_{\mathbb{X}}\nonumber\\&\quad+\left\|\mathcal{T}_{\gamma}(\tau_{2}-s_1)\left[h'_1(s_1,\tilde{x^n}(\tau_1^-))-h'_1(s_1,\tilde{x}(\tau_1^-))\right]\right\|_{\mathbb{X}}\nonumber\\&\quad+\left\|\int_{0}^{s_{1}}(s_{1}-s)^{\gamma-1}\mathcal{S}_{\gamma}(s_{1}-s)\left[f(s,\tilde{x^n}_{\varrho(s,\tilde{x^n_s})})-f(s,\tilde{x}_{\varrho(s,\tilde{x}_s)})\right]\mathrm{d}s\right\|_{\mathbb{X}}\nonumber\\&\quad+\left\|\int_{0}^{\tau_{2}}(\tau_{2}-s)^{\gamma-1}\mathcal{S}_{\gamma}(\tau_{2}-s)\left[f(s,\tilde{x^n}_{\varrho(s,\tilde{x^n_s})})-f(s,\tilde{x}_{\varrho(s,\tilde{x}_s)})\right]\mathrm{d}s\right\|_{\mathbb{X}}\nonumber\\&\quad+\left\|\int_{0}^{s_1}(s_1-s)^{\gamma-1}\mathcal{S}_{\gamma}(s_1-s)\mathrm{B}\left[u^{n,\alpha}_{0,\lambda}(s)-u^{\alpha}_{0,\lambda}(s)\right]\mathrm{d}s\right\|_{\mathbb{X}}\nonumber\\&\quad+\left\|\int_{0}^{s_1}(\tau_2-s)^{\gamma-1}\mathcal{S}_{\gamma}(\tau_2-s)\mathrm{B}\left[u^{n,\alpha}_{0,\lambda}(s)-u^{\alpha}_{0,\lambda}(s)\right]\mathrm{d}s\right\|_{\mathbb{X}} \nonumber\\&\leq M\left\|h_1(s_1,\tilde{x^n}(\tau_1^-))-h_1(s_1,\tilde{x}(\tau_1^-))\right\|_{\mathbb{X}}\nonumber\\&\quad+M(\tau_2-s)\left\|h'_1(s_1,\tilde{x^n}(\tau_1^-))-h'_1(s_1,\tilde{x}(\tau_1^-))\right\|_{\mathbb{X}}\nonumber\\&\quad+\frac{M}{\Gamma(2\gamma)}\int_{0}^{s_{1}}(s_{1}-s)^{2\gamma-1}\left\|f(s,\tilde{x^n}_{\varrho(s,\tilde{x^n_s})})-f(s,\tilde{x}_{\varrho(s,\tilde{x}_s)})\right\|_{\mathbb{X}}\mathrm{d}s\nonumber\\&\quad+\frac{M}{\Gamma(2\gamma))}\int_{0}^{\tau_{2}}(\tau_{2}-s)^{2\gamma-1}\left\|f(s,\tilde{x^n}_{\varrho(s,\tilde{x^n_s})})-f(s,\tilde{x}_{\varrho(s,\tilde{x}_s)})\right\|_{\mathbb{X}}\mathrm{d}s\nonumber\\&\quad+\frac{M\tilde{M}}{\Gamma(2\gamma)}\int_{0}^{s_1}(s_1-s)^{2\gamma-1}\left\|u^{n,\gamma}_{0,\lambda}(s)-u^{\gamma}_{0,\lambda}(s)\right\|_{\mathbb{U}}\mathrm{d}s\nonumber\\&\quad+\frac{M\tilde{M}}{\Gamma(2\gamma)}\int_{0}^{s_1}(\tau_2-s)^{\alpha-1}\left\|u^{n,\alpha}_{0,\lambda}(s)-u^{\alpha}_{0,\lambda}(s)\right\|_{\mathbb{U}}\mathrm{d}s\nonumber\\&\to0\ \mbox{ as }\ n\to\infty, 
		\end{align}
		where we have used the convergences \eqref{4.25}, \eqref{4.1}, Assumption \ref{as2.1} (\textit{H2}) and DCT. Moreover, similar to the convergence \eqref{4.1}, one can prove
		\begin{align*}
			&\left\|u^{n,\gamma}_{1,\lambda}(t)-u_{1,\lambda}^{\gamma}(t)\right||_{\mathbb{U}}\to 0\ \mbox{ as }\ n\to\infty, \ \mbox{uniformly for}\ t\in[s_1,\tau_{2}].
		\end{align*}
		Further, employing a similar analogy as above  for $j=2,\ldots,p,$ one obtains
		\begin{align*}
			&\left\|u^{n,\gamma}_{j,\lambda}(t)-u_{j,\lambda}^{\gamma}(t)\right||_{\mathbb{U}}\to 0\ \mbox{ as }\ n\to\infty, \mbox{ uniformly for }\ t\in[s_j,\tau_{j+1}], \ j=2,\ldots,p.
		\end{align*} 
		Hence, we have
		\begin{align}\label{4.29}
			\left\|u^{n,\alpha}_{\lambda}(t)-u_{\lambda}^{\alpha}(t)\right||_{\mathbb{U}}\to 0\ \mbox{ as }\ n\to\infty,\ \mbox{ uniformly for }\ t\in[s_j,\tau_{j+1}],\ j=0,1,\ldots,p.
		\end{align}
		Using the convergences \eqref{4.25}, \eqref{4.29} and DCT, we arrive at 
		\begin{align*}
			&	\left\|(\mathscr{F}_{\lambda}x^n)(t)-(\mathscr{F}_{\lambda}x)(t)\right\|_{\mathbb{X}}\nonumber\\&\leq\int_{0}^{t}(t-s)^{\gamma-1}\left\|\mathcal{S}_{\gamma}(t-s)\mathrm{B}\left[u^{n,\alpha}_{\lambda}(s)-u_{\lambda}^{\gamma}(s)\right]\right\|_{\mathbb{X}}\mathrm{d}s\nonumber\\&\quad+\int_{0}^{t}(t-s)^{\gamma-1}\left\|\mathcal{S}_{\gamma}(t-s)\left[f(s,\tilde{x^n}_{\varrho(s,\tilde{x^n_s})})-f(s,\tilde{x}_{\varrho(s,\tilde{x}_s)})\right]\right\|_{\mathbb{X}}\mathrm{d}s\nonumber\\&\le\frac{M\tilde{M}}{\Gamma(2\gamma)}\int_{0}^{t}(t-s)^{2\gamma-1}\left\|u^{n,\gamma}_{\lambda}(s)-u_{\lambda}^{\gamma}(s)\right\|_{\mathbb{U}}\mathrm{d}s\nonumber\\&\quad+\frac{M}{\Gamma(2\gamma))}\int_{0}^{t}(t-s)^{2\gamma-1}\left\|f(s,\tilde{x^n}_{\varrho(s,\tilde{x^n_s})})-f(s,\tilde{x}_{\varrho(s, \tilde{x}_s)})\right\|_{\mathbb{X}}\mathrm{d}s\nonumber\\&\to 0 \ \mbox{ as }\ n\to\infty, \ \mbox{ uniformly for }\ t\in[0,t_1].
		\end{align*}
		Similarly, for  $t\in(s_j,\tau_{j+1}],\ j=1,\ldots,p$, we deduce that 
		\begin{align*}
			&	\left\|(\mathscr{F}_{\lambda}x^n)(t)-(\mathscr{F}_{\lambda}x)(t)\right\|_{\mathbb{X}}\nonumber\\&\le\left\|\mathcal{C}_{\gamma}(t-s_j)\left[h_j(s_j,\tilde{x^n}(\tau_j^-))-h_j(s_j,\tilde{x}(\tau_j^-))\right]\right\|_{\mathbb{X}}\nonumber\\&\quad+\left\|\mathcal{T}_{\gamma}(t-s_j)\left[h'_j(s_j,\tilde{x^n}(\tau_j^-))-h'_j(s_j,\tilde{x}(\tau_j^-))\right]\right\|_{\mathbb{X}}\nonumber\\&\quad+\int_{0}^{s_j}(s_j-s)^{\gamma-1}\left\|\mathcal{S}_{\gamma}(s_j-s)\left[f(s,\tilde{x^n}_{\varrho(s, \tilde{x^n_s})})-f(s,\tilde{x}_{\varrho(s, \tilde{x}_s)})\right]\right\|_{\mathbb{X}}\mathrm{d}s\nonumber\\&\quad+\int_{0}^{s_j}(s_j-s)^{\gamma-1}\left\|\mathcal{S}_{\gamma}(s_j-s)\mathrm{B}\left[u^{n,\gamma}_{\lambda}(s)-u_{\lambda}^{\gamma}(s)\right]\right\|_{\mathbb{X}}\mathrm{d}s\nonumber\\&\quad+\int_{0}^{t}(t-s)^{\gamma-1}\left\|\mathcal{S}_{\gamma}(t-s)\mathrm{B}\left[u^{n,\gamma}_{\lambda}(s)-u_{\lambda}^{\gamma}(s)\right]\right\|_{\mathbb{X}}\mathrm{d}s\nonumber\\&\quad+\int_{0}^{t}(t-s)^{\gamma-1}\left\|\mathcal{S}_{\gamma}(t-s)\left[f(s,\tilde{x^n}_{\varrho(s,\tilde{x^n_s})})-f(s,\tilde{x}_{\varrho(s,\tilde{x}_s)})\right]\right\|_{\mathbb{X}}\mathrm{d}s\nonumber\\&\le M\left\|h_j(s_j,\tilde{x^n}(\tau_j^-))-h_j(s_j,\tilde{x}(\tau_j^-))\right\|_{\mathbb{X}}\nonumber\\&\quad+M(t-s_j)\left\|h'_j(s_j,\tilde{x^n}(\tau_j^-))-h'_j(s_j,\tilde{x}(\tau_j^-))\right\|_{\mathbb{X}}\nonumber\\&\quad+\frac{M}{\Gamma(2\gamma)}\int_{0}^{s_j}(s_j-s)^{2\gamma-1}\left\|f(s,\tilde{x^n}_{\varrho(s,\tilde{x^n_s})})-f(s,\tilde{x}_{\varrho(s, \tilde{x}_s)})\right\|_{\mathbb{X}}\mathrm{d}s\nonumber\\&\quad+ \frac{M\tilde{M}}{\Gamma(2\gamma)}\int_{0}^{s_j}(s_j-s)^{2\gamma-1}\left\|u^{n,\gamma}_{\lambda}(s)-u_{\lambda}^{\gamma}(s)\right\|_{\mathbb{U}}\mathrm{d}s\nonumber\\&\quad+\frac{M\tilde{M}}{\Gamma(2\gamma)}\int_{0}^{t}(t-s)^{2\gamma-1}\left\|u^{n,\gamma}_{\lambda}(s)-u_{\lambda}^{\gamma}(s)\right\|_{\mathbb{U}}\mathrm{d}s\nonumber\\&\quad+\frac{M}{\Gamma(2\gamma)}\int_{0}^{t}(t-s)^{2\gamma-1}\left\|f(s,\tilde{x^n}_{\varrho(s,\tilde{x^n_s})})-f(s,\tilde{x}_{\varrho(s,\tilde{x}_s)})\right\|_{\mathbb{X}}\mathrm{d}s\nonumber\\&\to 0 \ \mbox{ as }\ n\to\infty, \ \mbox{uniformly for}\ t\in(s_j,\tau_{j+1}].
		\end{align*}
		Moreover, for $t\in(\tau_j, s_j],\ j=1,\ldots,p$, using Assumption \ref{as2.1} (\textit{H2}), we obtain
		\begin{align*}
			\left\|(\mathscr{F}_{\lambda}x^n)(t)-(\mathscr{F}_{\lambda}x)(t)\right\|_{\mathbb{X}}&\le\left\|h_j(t,\tilde{x^n}(\tau_j^-))-h_j(t,\tilde{x}(\tau_j^-))\right\|_{\mathbb{X}}\to 0 \ \mbox{ as }\ n\to\infty.
		\end{align*}
		Therefore, it follows that $\mathscr{F}_{\lambda}$ is continuous.
		
		Hence, by the application of \emph{Schauder's fixed point theorem}, we conclude that the operator $\mathscr{F}_{\lambda}$ has a fixed point in $\mathcal{E}_{r}$, or the system \eqref{1.1} has a mild solution.
	\end{proof}
	
	In order to verify the approximate controllability of the system \eqref{1.1}, we replace the assumption (\textit{H2}) of the function $f(\cdot,\cdot)$ by the following assumption: 
	\begin{enumerate}\label{as}
		\item [\textit{(H3)}] The function $ f: J \times \mathfrak{B} \rightarrow \mathbb{X} $ satisfies Assumption (\textit{H1})(i) and there exists a function $ \phi\in \mathrm{L}^{\frac{1}{\delta}}(J;\mathbb{R}^+)$ with $\delta\in[0,\gamma]$ such that $$ \|f(t,\psi)\|_{\mathbb{X}}\leq \phi(t),\ \text{ for all }\  (t,\psi) \in J \times  \mathfrak{B}. $$ 
	\end{enumerate}
	\begin{theorem}\label{thm4.4}
		Let Assumptions (R1)-(R3), (H0)-(H1), (H3) and the condition \eqref{cnd} of Theorem \ref{thm4.3} be fulfilled. Then the system \eqref{1.1} is approximately controllable.
	\end{theorem}
	\begin{proof}
		Theorem \ref{thm4.3}, ensures that for every $\lambda>0$ and $\xi_j\in \mathbb{X},$ for $j=0,1,\ldots p$, the system \eqref{1.1} has a mild solution, say, $x^{\lambda}\in\mathcal{E}_r$ with the control $u_{\lambda}^{\gamma}$ same as given in \eqref{C}. Then, it is immediate that
		\begin{equation}\label{M}
			x^{\lambda}(t)=\begin{dcases}
				\mathcal{C}_{\gamma}(t)\psi(0)+\mathcal{T}_{\gamma}(t)\eta+\int_{0}^{t}(t-s)^{\gamma-1}\mathcal{S}_{\gamma}(t-s)\left[\mathrm{B}u_{\lambda}^{\gamma}(s)+f(s,\tilde{x^\lambda}_{\varrho(s, \tilde{x^\lambda_s})})\right]\mathrm{d}s,\\ \qquad\qquad\qquad\qquad\qquad\qquad\qquad\qquad\qquad\qquad t\in[0, \tau_1],\\
				h_j(t, \tilde{x^\lambda}(\tau_j^-)),\qquad\qquad\qquad\qquad\qquad\qquad\qquad t\in(\tau_j, s_j],\ j=1,\ldots,p,\\
				\mathcal{C}_{\gamma}(t-s_j)h_j(s_j, \tilde{x^\lambda}(\tau_j^-))+\mathcal{T}_{\gamma}(t-s_j)h'_j(s_j, \tilde{x^\lambda}(\tau_j^-))\\ \quad -\int_{0}^{s_j}(s_j-s)^{\gamma-1}\mathcal{S}_{\gamma}(s_j-s)\left[\mathrm{B}u_{\lambda}^{\gamma}(s)+f(s,\tilde{x^\lambda}_{\varrho(s, \tilde{x^\lambda_s})})\right]\mathrm{d}s \\\qquad+\int_{0}^{t}(t-s)^{\gamma-1}\mathcal{S}_{\gamma}(t-s)\left[\mathrm{B}u_{\lambda}^{\gamma}(s)+f(s,\tilde{x^\lambda}_{\varrho(s, \tilde{x^\lambda_s})})\right]\mathrm{d}s,\\ \qquad\qquad\qquad\qquad\qquad\qquad\qquad\qquad\qquad\qquad t\in(s_j,\tau_{j+1}],\ j=1,\ldots,p.
			\end{dcases}
		\end{equation}
		Next, we estimate
		\begin{align}\label{4.35}
			x^{\lambda}(T)&=\mathcal{C}_{\gamma}(T-s_p)h_p(s_p,\tilde{x^\lambda}(\tau_p^-))+\mathcal{T}_{\gamma}(T-s_p)h_{p}(s_p,\tilde{x^\lambda}(\tau_p^-))\nonumber\\&\quad-\int_{0}^{s_p}(s_p-s)^{\gamma-1}\mathcal{S}_{\gamma}(s_p-s)\left[\mathrm{B}u^{\gamma}_{\lambda}(s)+f(s,\tilde{x^\lambda}_{\varrho(s,\tilde{x^\lambda_s})})\right]\mathrm{d}s\nonumber\\&\quad+\int_{0}^{T}(T-s)^{\gamma-1}\mathcal{S}_{\gamma}(T-s)\left[\mathrm{B}u^{\gamma}_{\lambda}(s)+f(s,\tilde{x^\lambda}_{\varrho(s,\tilde{x^\lambda_s})})\right]\mathrm{d}s\nonumber\\&=\mathcal{C}_{\gamma}(T-s_p)h_p(s_p,\tilde{x^\lambda}(\tau_p^-))+\mathcal{T}_{\gamma}(T-s_p)h_{p}(s_p,\tilde{x^\lambda}(\tau_p^-))\nonumber\\&\quad-\int_{0}^{s_p}(s_p-s)^{\gamma-1}\mathcal{S}_{\gamma}(s_p-s)\left[\mathrm{B}u^{\gamma}_{\lambda}(s)+f(s,\tilde{x^\lambda}_{\varrho(s,\tilde{x^\lambda_s})})\right]\mathrm{d}s\nonumber\\&\quad+\int_{0}^{T}(T-s)^{\gamma-1}\mathcal{S}_{\gamma}(T-s)f(s,\tilde{x^\lambda}_{\varrho(s,\tilde{x^\lambda_s})})+\int_{0}^{s_p}(T-s)^{\gamma-1}\mathcal{S}_{\gamma}(T-s)\mathrm{B}u^{\gamma}_{\lambda}(s)\mathrm{d}s\nonumber\\&\quad+\int_{s_p}^{T}(T-s)^{\gamma-1}\mathcal{S}_{\gamma}(T-s)\mathrm{B}\mathrm{B}^*\mathcal{S}_{\gamma}(T-s)^*\mathscr{J}\left[\mathcal{R}(\lambda,\Phi_{s_p}^{T})g_p(x^\lambda(\cdot))\right]\mathrm{d}s\nonumber\\&=\mathcal{C}_{\gamma}(T-s_p)h_p(s_p,\tilde{x^\lambda}(\tau_p^-))+\mathcal{T}_{\gamma}(T-s_p)h_{p}(s_p,\tilde{x^\lambda}(\tau_p^-))\nonumber\\&\quad-\int_{0}^{s_p}(s_p-s)^{\gamma-1}\mathcal{S}_{\gamma}(s_p-s)\left[\mathrm{B}u^{\gamma}_{\lambda}(s)+f(s,\tilde{x^\lambda}_{\varrho(s,\tilde{x^\lambda_s})})\right]\mathrm{d}s\nonumber\\&\quad+\int_{0}^{T}(T-s)^{\gamma-1}\mathcal{S}_{\gamma}(T-s)f(s,\tilde{x^\lambda}_{\varrho(s,\tilde{x^\lambda_s})})+\int_{0}^{s_p}(T-s)^{\gamma-1}\mathcal{S}_{\gamma}(T-s)\mathrm{B}u^{\gamma}_{\lambda}(s)\mathrm{d}s\nonumber\\&\quad+\Phi_{s_p}^{T}\mathscr{J}\left[\mathcal{R}(\lambda,\Phi_{s_p}^{T})g_p(x^\lambda(\cdot))\right]\nonumber\\&=\xi_p-\lambda\mathcal{R}(\lambda,\Phi_{s_p}^{T})g_p(x^\lambda(\cdot)).
		\end{align}
		The sequence $x^{\lambda}(t)$ is bounded in $\mathbb{X},$ for each $t\in J$, follows from the argument $ x^{\lambda}\in\mathcal{E}_r $. Then by the Banach-Alaoglu theorem, we can find a subsequence, still labeled as $ x^{\lambda}$ such that 
		\begin{align*}
			x^\lambda(t)\xrightharpoonup{w}z(t) \ \mbox{ in }\ \mathbb{X} \ \ \mbox{ as }\  \ \lambda\to0^+,\ t\in J.
		\end{align*}
		Using the condition (\textit{H2}) of Assumption \ref{as2.1}, we have the following convergence:
		\begin{align}
			h_p(t,x^\lambda(\tau_p^-))\to h_p(t,z(\tau_p^-)) \ \mbox{ in }\ \mathbb{X} \ \mbox{ as }\   \lambda\to0^+, \ \mbox{ for all }\ t\in J,\label{4.19}\\
			h'_p(t,x^\lambda(\tau_p^-))\to h'_p(t,z(\tau_p^-)) \ \mbox{ in }\ \mathbb{X} \ \mbox{ as }\   \lambda\to0^+, \ \mbox{ for all }\ t\in J\label{4.20}.
		\end{align}
		Moreover, by using Assumption  \textit{(H3)}, we obtain
		\begin{align}
			\int_{t_1}^{t_2}\left\|f(s,\tilde{x^{\lambda}}_{\varrho(s,\tilde{x^{\lambda}_s})})\right\|_{\mathbb{X}}^{2}\mathrm{d}s&\le \int_{t_1}^{t_2}\phi^2(s)\mathrm{d} s\leq \left(\int_{t_1}^{t_2}\phi^{\frac{1}{\delta}}(s)\mathrm{d}s\right)^{2\delta}(t_1-t_2)^{1-2\delta}<+\infty, \nonumber
		\end{align}
		for any $ t_1,t_2\in[0,T]$ with $t_1<t_2$. The above estimate guarantees that the  sequence $ \{f(\cdot, \tilde{x^{\lambda}}_{\varrho(s,\tilde{x^{\lambda}_s})}): \lambda >0\}$  in $ \mathrm{L}^2([t_1,t_2]; \mathbb{X})$ is bounded. Once again by applying the Banach-Alaoglu theorem, we can extract a subsequence still denoted by$ \{f(\cdot, \tilde{x^{\lambda}}_{\rho(s,\tilde{x^{\lambda}}_s)}): \lambda > 0 \}$ such that 
		\begin{align}\label{4.36}
			f(\cdot, \tilde{x^{\lambda}}_{\rho(s,\tilde{x^{\lambda}}_s)})\xrightharpoonup{w}f(\cdot) \ \mbox{ in }\ \mathrm{L}^2([t_1,t_2];\mathbb{X}) \ \mbox{ as }\  \lambda\to0^+.
		\end{align}
		Next, we compute  
		\begin{align*}
			&\int_{0}^{s_p}\left\|u^{\gamma}_{\lambda}(s)\right\|^2_{\mathbb{U}}\mathrm{d}s\nonumber\\&=\int_{0}^{\tau_1}\left\|u^{\gamma}_{\lambda}(s)\right\|^2_{\mathbb{U}}\mathrm{d}s+\int_{\tau_1}^{s_1}\left\|u^{\gamma}_{\lambda}(s)\right\|^2_{\mathbb{U}}\mathrm{d}s+\int_{s_1}^{\tau_2}\left\|u^{\gamma}_{\lambda}(s)\right\|^2_{\mathbb{U}}\mathrm{d}s+\cdots+\int_{\tau_p}^{s_p}\left\|u^{\gamma}_{\lambda}(s)\right\|^2_{\mathbb{U}}\mathrm{d}s\nonumber\\&=\int_{0}^{\tau_1}\left\|u^{\gamma}_{0,\lambda}(s)\right\|^2_{\mathbb{U}}\mathrm{d}s+\int_{s_1}^{\tau_2}\left\|u^{\gamma}_{1,\lambda}(s)\right\|^2_{\mathbb{U}}\mathrm{d}s+\cdots+\int_{s_{p-1}}^{\tau_p}\left\|u^{\gamma}_{p-1,\lambda}(s)\right\|^2_{\mathbb{U}}\mathrm{d}s\nonumber\\&\le\left(\frac{M\tilde{M}}{\lambda\Gamma(2\gamma)}\right)^{2}\frac{T^{2\gamma+1}}{2\gamma+1}\sum_{j=0}^{p-1}C_j^2=C,
		\end{align*}
		where $C_j,$ for $j=1,\ldots,p-1$ are the same as given in \eqref{4.5} and $C_0=N_0$ given in \eqref{4.4}. Moreover, the above fact implies that the sequence $\{u^\gamma_{\lambda}(\cdot): \lambda >0\}$  in $ \mathrm{L}^2([0,s_p]; \mathbb{U})$ is bounded. Further, by an application of the Banach-Alaoglu theorem, we obtain  a subsequence, still denoted by $\{u^\gamma_{\lambda}(\cdot): \lambda >0\}$ such that 
		\begin{align}\label{4}
			u^\gamma_{\lambda}(\cdot)\xrightharpoonup{w}u^\gamma(\cdot) \ \mbox{ in }\ \mathrm{L}^2([0,s_p];\mathbb{U}) \ \mbox{ as }\  \lambda\to0^+.
		\end{align}
		Next, we compute
		\begin{align}\label{4.37}
			&	\left\|g_{p}(x^{\lambda}(\cdot))-\omega\right\|_{\mathbb{X}}\nonumber\\&\le\left\|\mathcal{C}_{\gamma}(T-s_p)(h_p(s_p,\tilde{x^\lambda}(\tau_p^-))-h_p(s_p,z(\tau_p^-)))\right\|_{\mathbb{X}}\nonumber\\&\quad+\left\|\mathcal{T}_{\gamma}(T-s_p)(h'_p(s_p,\tilde{x^\lambda}(\tau_p^-))-h'_p(s_p,z(\tau_p^-)))\right\|_{\mathbb{X}}\nonumber\\&\quad+\left\|\int_{0}^{s_p}(s_p-s)^{\gamma-1}\mathcal{S}_{\gamma}(s_p-s)\mathrm{B}\left[u^\gamma_{\lambda}(s)-u^{\gamma}(s)\right]\mathrm{d}s\right\|_{\mathbb{X}}\nonumber\\&\quad+\left\|\int_{0}^{s_p}(s_p-s)^{\gamma-1}\mathcal{S}_{\gamma}(s_p-s)\left[f(s,\tilde{x^\lambda}_{\varrho(s,\tilde{x^\lambda_s})})-f(s)\right]\mathrm{d}s\right\|_{\mathbb{X}}\nonumber\\&\quad+\left\|\int_{0}^{s_p}(T-s)^{\gamma-1}\mathcal{S}_{\gamma}(T-s)\mathrm{B}\left[u^\gamma_{\lambda}(s)-u^{\gamma}(s)\right]\mathrm{d}s\right\|_{\mathbb{X}}\nonumber\\&\quad+\left\|\int_{0}^{T}(T-s)^{\gamma-1}\mathcal{S}_{\gamma}(T-s)\left[f(s,\tilde{x^\lambda}_{\varrho(s,\tilde{x^\lambda_s})})-f(s)\right]\mathrm{d}s\right\|_{\mathbb{X}}\nonumber\\&\le\left\|\mathcal{C}_{\gamma}(T-s_p)(h_p(s_p,x^\lambda(\tau_p^-))-h_p(s_p,z(\tau_p^-)))\right\|_{\mathbb{X}}\nonumber\\&\quad+\left\|\mathcal{T}_{\gamma}(T-s_p)(h'_p(s_p,x^\lambda(\tau_p^-))-h'_p(s_p,z(\tau_p^-)))\right\|_{\mathbb{X}}\nonumber\\&\quad+\left\|\int_{0}^{s_p}(s_p-s)^{\gamma-1}\mathcal{S}_{\gamma}(s_p-s)\mathrm{B}\left[u^\gamma_{\lambda}(s)-u^{\gamma}(s)\right]\mathrm{d}s\right\|_{\mathbb{X}}\nonumber\\&\quad+\left\|\int_{0}^{s_p}(s_p-s)^{\gamma-1}\mathcal{S}_{\gamma}(s_p-s)\left[f(s,\tilde{x^\lambda}_{\varrho(s,\tilde{x^\lambda_s})})-f(s)\right]\mathrm{d}s\right\|_{\mathbb{X}}\nonumber\\&\quad+\frac{T^{2\gamma-1}-(T-s_p)^{2\gamma-1}}{2\gamma-1}\left(\int_{0}^{s_p}\left\|\mathcal{S}_{\gamma}(T-s)\mathrm{B}\left[u^\gamma_{\lambda}(s)-u^{\gamma}(s)\right]\right\|^2_{\mathbb{X}}\mathrm{d}s\right)^{\frac{1}{2}}\nonumber\\&\quad+\left\|\int_{0}^{T}(T-s)^{\gamma-1}\mathcal{S}_{\gamma}(T-s)\left[f(s,\tilde{x^\lambda}_{\varrho(s,\tilde{x^\lambda_s})})-f(s)\right]\mathrm{d}s\right\|_{\mathbb{X}}\nonumber\\&\to 0\ \mbox{ as }\ \lambda\to0^+, 
		\end{align}
		where 
		\begin{align*}
			\omega &=\xi_p-\mathcal{C}_{\gamma}(T-s_p)h_p(s_p,z(\tau_p^-))-\mathcal{T}_{\gamma}(T-s_p)h'_p(s_p,z(\tau_p^-))\nonumber\\&\quad-\int_{0}^{T}(T-s)^{\gamma-1}\mathcal{S}_{\gamma}(T-s)f(s)\mathrm{d}s+\int_{0}^{s_p}(s_p-s)^{\gamma-1}\mathcal{S}_{\gamma}(s_p-s)\left[\mathrm{B}u^{\gamma}(s)+f(s)\right]\mathrm{d}s\nonumber\\&\quad-\int_{0}^{s_p}(T-s)^{\gamma-1}\mathcal{S}_{\gamma}(T-s)\mathrm{B}u^{\gamma}(s)\mathrm{d}s.
		\end{align*}
		In \eqref{4.37}, we have used the convergences \eqref{4.19}, \eqref{4.20}, \eqref{4.36}, \eqref{4}, DCT  and the compactness of the operator $(\mathrm{Q}f)(\cdot) =\int_{0}^{\cdot}(\cdot-s)^{\gamma-1}\mathcal{S}_{\gamma}(\cdot-s)f(s)\mathrm{d}s:\mathrm{L}^2(J;\mathbb{X})\rightarrow \mathrm{C}(J;\mathbb{X}) $ (see Lemma \ref{lem2.12}).
		Finally, by using the equality \eqref{4.35}, we estimate
		\begin{align}
			\left\|x^{\lambda}(T)-\xi_p\right\|_{\mathbb{X}}&\le\left\|\lambda\mathcal{R}(\lambda,\Phi_{s_p}^{T})g_p(x(\cdot))\right\|_{\mathbb{X}}\nonumber\\&\le\left\|\lambda\mathcal{R}(\lambda,\Phi_{s_p}^{T})(g_p(x(\cdot))-\omega)\right\|_{\mathbb{X}}+\left\|\lambda\mathcal{R}(\lambda,\Phi_{s_p}^{T})\omega\right\|_{\mathbb{X}}\nonumber\\&\le\left\|\lambda\mathcal{R}(\lambda,\Phi_{s_p}^{T})\right\|_{\mathcal{L}(\mathbb{X})}\left\|g_p(x(\cdot))-\omega\right\|_{\mathbb{X}}+\left\|\lambda\mathcal{R}(\lambda,\Phi_{s_p}^{T})\omega\right\|_{\mathbb{X}}.
		\end{align}
		Using the above inequality, \eqref{4.37} and Assumption \ref{as2.1} (\textit{H0}), we obtain
		\begin{align*}
			\left\|x^{\lambda}(T)-\xi_p\right\|_{\mathbb{X}}\to0\ \mbox{ as }\ \lambda\to0^+.
		\end{align*}
		Hence, the system \eqref{1.1} is approximately controllable on $J$.
	\end{proof}
	
	\section{Application}\label{app}\setcounter{equation}{0}
	In this section, we investigate the approximate controllability of the fractional order wave equation with non-instantaneous impulses and delay:
	
	\begin{Ex}\label{ex1} Let us consider the following fractional system:
		\begin{equation}\label{ex}
			\left\{
			\begin{aligned}
				\frac{\partial^\alpha v(t,\xi)}{\partial t^\alpha}&=\frac{\partial^2v(t,\xi)}{\partial \xi^2}+w(t,\xi)+\int_{-\infty}^{t}b(t-s)v(s-\beta(\|v(t)\|),\xi)\mathrm{d}s, \\&\qquad \qquad \qquad\qquad \ t\in\bigcup_{j=0}^{m} (s_j, \tau_{j+1}]\subset J=[0,T], \ \xi\in[0,\pi], \\
				v(t,\xi)&=h_j(t,v(t_j^-,\xi)),\ \ \ t\in(\tau_j,s_j],\ j=1,\ldots, m,\ \xi\in[0,\pi],\\
				\frac{\partial v(t,\xi)}{\partial t}&=\frac{\partial h_j(t,v(\tau_j^-,\xi))}{\partial t},\ t\in(\tau_j,s_j],\ j=1\ldots,m,\ \xi\in[0,\pi],\\
				v(t,0)&=0=v(t,\pi), \qquad \  t\in J, \\
				v(\theta,\xi)&=\psi(\theta,\xi), \frac{\partial v(0,\xi)}{\partial t}=\zeta_0(\xi),\ \xi\in[0,\pi], \ \theta\leq0,
			\end{aligned}
			\right.
		\end{equation}
		where $\alpha\in(1,2)$ and the function $w:J\times[0,\pi]\to[0,\pi]$ is square integrable in $t$ and $\xi$, and the functions $\beta:[0,\infty)\to[0,\infty)$ and $b:[0, \infty)\rightarrow\mathbb{R}$ are also continuous. The functions $\psi(\cdot,\cdot)$ and $h_j$ for $j=1,\ldots,p$ satisfy suitable conditions, which will be specified later.
	\end{Ex}
	
	\vskip 0.1 in
	\noindent\textbf{Step 1:} \emph{Strongly continuous families and phase space:} Let $\mathbb{X}_p= \mathrm{L}^{p}([0,\pi];\mathbb{R})$, for $p\in[2,\infty)$, and $\mathbb{U}=\mathrm{L}^{2}([0,\pi];\mathbb{R})$. Note that  $\mathbb{X}_p$ is a separable reflexive Banach space with strictly convex dual $\mathbb{X}_p^*=\mathrm{L}^{\frac{p}{p-1}}([0,\pi];\mathbb{R})$ and $\mathbb{U}$ is separable. We define the operator $\mathrm{A}_p:\mathrm{D}(\mathrm{A}_p)\subset\mathbb{X}_p\rightarrow \mathbb{X}_p$ as
	\begin{align}\label{5.9}
		\mathrm{A}_pf(\xi)=f''(\xi), \ \text{ where }\ \mathrm{D}(\mathrm{A}_p)= \mathrm{W}^{2,p}([0,\pi];\mathbb{R})\cap\mathrm{W}_0^{1,p}([0,\pi];\mathbb{R}).\end{align}
	Moreover, the spectrum of the operator $\mathrm{A}_p$ is given by $\sigma(\mathrm{A}_p)=\{-n^2:n\in\mathbb{N}\}$. Then, for every $f\in\mathrm{D}(\mathrm{A}_p)$, the operator $\mathrm{A}_p$ can be written as
	\begin{align*}
		\mathrm{A}_pf&= \sum_{n=1}^{\infty}-n^{2}\langle f, w_{n} \rangle  w_{n},\ \langle f,w_n\rangle :=\int_0^{\pi}f(\xi)w_n(\xi)\mathrm{d}\xi,
	\end{align*}
	where $w_n(\xi)=\sqrt{\frac{2}{\pi}}\sin(n\xi)$ are the normalized eigenfunctions  (with respect to the $\mathbb{X}_2$ norm) of the operator $\mathrm{A}_p$ corresponding to the eigenvalues $-n^2$, $n\in\mathbb{N}$. The operator  $\mathrm{A}_p$ satisfies all the conditions (\textit{R1})-(\textit{R3}) of Assumption \ref{ass2.1} (see application section of \cite{SS}). By applying  Lemma \ref{lem2.1}, we deduce the existence of a strongly continuous cosine family $\mathrm{C}_p(t),\  t \in \mathbb{R}$ in $\mathbb{X}_p$. The compactness of  the associated strongly continuous sine family $\mathrm{S}_p(t), t\in \mathbb{R},$  follows by  Lemma \ref{lem2.2}. The strongly continuous families $\{\mathrm{C}_p(t):t\in\mathbb{R}\}$ and $\{\mathrm{S}_p(t):t\in\mathbb{R}\}$ can be written as
	\begin{align*}
		\mathrm{C}_p(t)f&=\sum_{n=1}^{\infty}\cos(nt)\langle f, w_{n} \rangle  w_{n},\ f\in\mathbb{X}_p,\nonumber\\
		\mathrm{S}_p(t)f&=\sum_{n=1}^{\infty}\frac{1}{n}\sin(nt)\langle f, w_{n} \rangle  w_{n},\ f\in\mathbb{X}_p.\nonumber
	\end{align*}
	Next, we define the operator
	\begin{align*}
		\mathcal{S}_{\gamma,p}(t)f&=\int_{0}^{\infty} \gamma\theta\mathrm{M}_{\gamma}(\theta)\sum_{n=1}^{\infty}\frac{1}{n}\sin(nt^{\gamma}\theta)\langle f, w_{n} \rangle w_{n}\mathrm{d}\theta,\ f\in\mathbb{X}_p,  
	\end{align*}
	where $\gamma=\frac{\alpha}{2}$.
	\vskip 0.2 cm
	\noindent \emph{Phase space:} The space $\mathfrak{B}=\mathcal{PC}_{g}(\mathbb{X}_p)$ (see, Example \ref{exm2.8})  is a phase space satisfying the axioms (A1) and (A2) with $\mathcal{P}(t) =t$ and $\mathcal{Q}(t)= \mathcal{O}(t)$, (see, section 5 of  \cite{SS}). Specifically, one can choose the function $g(\theta)=e^{a\theta}$, for $a<0$. In this case, we assume the following condition:
	\begin{itemize}
		\item [$(C1)$]  The function $\psi\in\mathcal{PC}_{g}(\mathbb{X})$, we assume $L:=\esssup\limits_{\theta\in(-\infty,0]}|b(-\theta)|g(\theta)$.
	\end{itemize}
	\vskip 0.1 cm 
	\noindent\textbf{Step 2:} \emph{Approximate controllability.}
	Let us define $$x(t)(\xi):=v(t,\xi),\ \mbox{ for }\ t\in J\ \mbox{ and }\ \xi\in[0,\pi],$$ and the operator $\mathrm{B}:\mathbb{U}\to\mathbb{X}_p$ as  $$\mathrm{B}u(t)(\xi):=w(t,\xi)=\int_{0}^{\pi}K(\zeta,\xi)u(t)(\zeta)\mathrm{d}\zeta, \ t\in J,\ \xi\in [0,\pi],$$ where $K\in\mathrm{C}([0,\pi]\times[0,\pi];\mathbb{R})$ is a symmetric kernel, that is, $K(\zeta,\xi)=K(\xi,\zeta),$ for all $\zeta,\xi\in [0,\pi]$. We assume that the operator $\mathrm{B}$ is one-one. Hence, the operator $\mathrm{B}$ is bounded (see application section of \cite{SMJ}). The symmetry of the kernel ensures that the operator $\mathrm{B}$ is self-adjoint, that is,  $\mathrm{B}=\mathrm{B}^*$ . For example, one can take $K(\xi,\zeta)=e^{(\xi+\zeta)},\ \mbox{for all}\ \xi, \zeta\in [0,\pi]$.
	The function $\psi:(-\infty,0]\rightarrow\mathbb{X}$ is given as
	\begin{align}
		\nonumber \psi(t)(\xi)=\psi(t,\xi),\ \xi\in[0,\pi].
	\end{align}	 
	Next, the functions $f, \varrho:J\times \mathfrak{B}\to\mathbb{X}$ are defined as
	\begin{align}
		\nonumber f(t,\psi)\xi&:=\int_{-\infty}^{0}b(-\theta)\psi(\theta,\xi)\mathrm{d}\theta,\\
		\nonumber\varrho(t,\psi)&:=t-\beta(\|\psi(0)\|_{\mathbb{X}}),
	\end{align}	
	for $\xi\in[0,\pi]$. It is easy to verify that the function $f$ is continuous and uniformly bounded by $L$. Hence, the function $f$ fulfills the condition \textit{$(H1)$} of Assumption \ref{as2.1} and the condition \textit{$(H3)$}.
	
	Moreover, the impulse functions $h_{j}:[\tau_j,t_j]\times\mathbb{X}_p\to\mathbb{X}_p,$ for $j=1,\ldots,m,$ are defined as 
	\begin{align*}
		h_{j}(t,x)\xi:=\int_{0}^{\pi}\rho_j(t,\xi,z)\cos^2(x(\tau_j^-)z)\mathrm{d}z, \ \mbox{ for }\ t\in(\tau_j,s_j],
	\end{align*}
	where, $\rho_j\in\mathrm{C}^{1}(J\times[0,\pi]^2;\mathbb{R})$. Clearly, the impulses $h_{j}$ for $j=1,\ldots,m,$ satisfy the condition \textit{$(H2)$} of Assumption \ref{as2.1}.

	Using the above transformations, the system \eqref{ex} can be rewritten in the form \eqref{1.1} which satisfies Assumption \ref{as2.1} \textit{(H1)-(H2)} and Assumption (\textit{H3}). Next, we show that  the corresponding linear fractional control system of the equation \eqref{1.1} is approximately controllable. To achieve this, we assume that
	$$\mathrm{B}^*\mathcal{S}_{\gamma,p}(T-t)^*x^*=0,\ \mbox{ for any}\ x^*\in\mathbb{X}_p^*,\ 0\le t<T.$$ Since the operator $\mathrm{B}^*$ is one-one, we get that
	\begin{align}\label{5.3}
		\mathcal{S}_{\gamma,p}(T-t)^*x^*=0,\ \mbox{ for all } \ t\in [0,T).
	\end{align}
	Let us compute
	\begin{align}\label{5.4}
		\lim_{t\downarrow T}\frac{\mathcal{S}_{\gamma,p}(T-t)^*x^*}{(T-t)^\gamma}&=	\lim_{t\downarrow T}\frac{\int_{0}^{\infty}\gamma\theta\mathrm{M}_{\gamma}(\theta)\sum_{n=1}^{\infty}\frac{1}{n}\sin(n(T-t)^{\gamma}\theta)\langle x^*, w_{n} \rangle w_{n}\mathrm{d}\theta}{(T-t)^{\gamma}}\nonumber\\&=\lim_{t\downarrow T}\sum_{n=1}^{\infty}\frac{1}{n}\gamma\int_{0}^{\infty}\theta\mathrm{M}_{\gamma}(\theta)\left(\frac{\sin(n(T-t)^{\gamma}\theta)}{(T-t)^\gamma}\right)\mathrm{d}\theta\langle x^*, w_{n} \rangle w_n \nonumber\\&=\sum_{n=1}^{\infty}\frac{1}{n}\gamma\int_{0}^{\infty}\theta\mathrm{M}_{\gamma}(\theta)\left(\lim_{t\downarrow T}\frac{\sin(n(T-t)^{\gamma}\theta)}{(T-t)^\gamma}\right)\mathrm{d}\theta\langle x^*, w_{n} \rangle w_n\nonumber\\&=\sum_{n=1}^{\infty} \gamma  \int_{0}^{\infty} \theta^2\mathrm{M}_{\gamma}\mathrm{d}\theta\langle x^*, w_{n} \rangle w_n\nonumber\\&=\frac{2\gamma}{\Gamma(1+2\gamma)}\sum_{n=1}^{\infty}\langle x^*, w_{n} \rangle w_n,
	\end{align}
	where $\gamma\in(0,1)$. Combining the estimates \eqref{5.3} and \eqref{5.4}, we obtain
	\begin{align*}
		\sum_{n=1}^{\infty}\langle x^*, w_{n} \rangle w_n=0,
	\end{align*}
	which implies that $x^*=0$.	Thus, by applying Lemma \ref{lem3.4} and Remark \ref{rem3.4}, we deduce that the corresponding linear fractional control system of \eqref{1.1} is approximately controllable. Finally, by Theorem \ref{thm4.4}, we conclude that the semilinear fractional control system \eqref{1.1} (equivalent to the system \eqref{ex}) is approximately controllable.
	
	\section{Conclusions} In this manuscript, we considered the semilinear fractional control system \eqref{1.1} of order $1<\alpha<2$, where the operator $\mathrm{A}$ generates a strongly continuous cosine family $\{\mathrm{C}(t):t\in\mathbb{R}\}$. We first formulated the linear regulator problem and obtained the optimal control in the feedback form. With the help of this optimal control, we discussed the approximate controllability of the linear control system \eqref{3.2}. Then, we proved the existence of a mild solution of the semilinear fractional control system \eqref{1.1} for the suitable control function defined in \eqref{C} via the resolvent operator and the Schauder fixed point theorem. Then, we derived sufficient conditions for the approximate controllability of the system \eqref{1.1}, whenever the corresponding linear control system is approximately controllable. Moreover, we modified the axioms of phase space to deal with impulsive functional differential equations. Note that, in this work, we have used the idea of $\alpha/2$-resolvent family related to cosine family generated by the operator $\mathrm{A}$. 
	
	The existence of solutions and the approximate controllability results of the semilinear fractional abstract Cauchy problem of order $\alpha\in(1,2)$ by using the idea related to the $\alpha$-resolvent family introduced in \cite{JFA} is not investigated so far, and it will be  addressed in a  future work.  
	
	\medskip\noindent
	{\bf Acknowledgments:} S. Arora would like first to thank the Council of Scientific and Industrial Research, New Delhi, Government of India (File No. 09/143(0931)/2013 EMR-I), for financial support to carry out his research work and also thank the Department of Mathematics, Indian Institute of Technology Roorkee (IIT Roorkee), for providing stimulating scientific environment and resources. M. T. Mohan would like to thank the Department of Science and Technology (DST), Govt of India for Innovation in Science Pursuit for Inspired Research (INSPIRE) Faculty Award (IFA17-MA110). J. Dabas would like to thank the Council of Scientific and Industrial Research, New Delhi, Government of India, project (ref. no. 25(0315)/20/EMR-II).


\begin{thebibliography}{10}
		
		\bibitem{SAM} S. Arora, M.T. Mohan and J. Dabas, Approximate controllability of a Sobolev type impulsive functional evolution system in Banach spaces, \emph{Math. Control Relat. Fields}, (2020), \url{https://doi:10.3934/mcrf.2020049}.
		
		\bibitem{SMJ} S. Arora, M.T. Mohan and J. Dabas, Approximate controllability of fractional order non-instantaneous impulsive functional evolution equations with state-dependent delay in Banach spaces, \url{https://arxiv.org/abs/2106.02939v1}.
		
		\bibitem{SM} S. Arora, M.T. Mohan and J. Dabas, Approximate controllability of the non-autonomous impulsive evolution equation with state-dependent delay in Banach spaces, \emph{Nonlinear Anal. Hybrid Syst.}, {\bf 39} (2021), 100989.
		
		\bibitem{AO} O. Arino, K. Boushaba and A. Boussouar, A mathematical model of the dynamics of the phytoplankton-nutrient system. Spatial hetrogeneity in ecological models, {\em Nonlinear Anal. Real World Appl.}, {\bf 1} (2000), 69-87.
		
		\bibitem{AA} E. Asplund, \emph{Averaged norms}, Israel J. Math., {\bf 5} (1967), 227-233.
		
		\bibitem{VB} V. Barbu, \emph{Analysis and Control of Nonlinear Infinite Dimensional Systems}, Academic Press, New York, 1993.
		
		\bibitem{VB1}  V. Barbu, \emph{Controllability and Stabilization of Parabolic Equations}, Springer International Publishing AG,  2018. 
		
		\bibitem{BM} M. Benchohra, J. Henderson and S. Ntouyas, {\em Impulsive Differential Equations and Inclusions}, Hindawi Publishing Corporation, New York 2006.
		
		\bibitem{FC} F. Chen, D. Sun and J. Shi, Periodicity in a food-limited population model with toxicants and state dependent delays, {\em J. Math. Anal. Appl.}, {\bf 288} (2003), 136-146.
		
		\bibitem{DNC} D.N. Chalishajar, K. Malar and K. Karthikeyan, Approximate controllability of abstract impulsive fractional neutral evolution equations with infinite delay in Banach spaces, \emph{Electron. J. Differ. Equ.}, {\bf 275} (2013), 1–21 .
		
		\bibitem{JA} J. dabs, A. chauhan and M. Kumar, Existence of the mild solutions for impulsive fractional equations with infinite delay, \emph{Int. J. Differ. Equ.}, {\bf 2011} (2011), 793023.
		
		\bibitem{GDJZ} G. Da Prato and J. Zabczyk, \emph{Ergodicity for Infinite Dimensional Systems}, London Mathematical Society Lecture Notes, Cambridge University Press, 1996.
		
		\bibitem{EI} I. Ekeland and T. Turnbull, \emph{Infinite Dimensional Optimization and Convexity}, Chicago press, London, 1983.
		
		\bibitem{MFb} M. Fabian et.al.,	\emph{Functional Analysis and Infinite Dimensional Geometry}, CMS Books in Mathematics, Springer-Verlag, New York, 2001.
		
		
		\bibitem{GR} G.R. Gautam and J. Dabas, Mild solutions for class of neutral fractional functional differential equations with not instantaneous impulses, \emph{Appl. Math. Comput.}, {\bf 259} (2015), 480-489.
		
		
		\bibitem{AGK} A. Grudzka and K. Rykaczewski, On approximate controllability of functional impulsive evolution inclusions in a Hilbert space, \emph{J. Optim. Theory Appl.}, {\bf 166} (2015), 414-439.
		
		\bibitem{GU} L. Guedda, Some remarks in the study of impulsive differential equations and inclusions with delay, \emph{Fixed Point Theory}, {\bf 12} (2011), 349–354.
		
		\bibitem{JFA} H.R. Henr\'iquez, J. G. Mesquita, J. C.Pozo, Existence of solutions of the abstract Cauchy problem of fractional order, \emph{J. Funct. Anal.}, (2021), 109028, doi: https://doi.org/10.1016/j.jfa.2021.109028.
		
		\bibitem{Jh} J.W. He, Y. Liang, B. Ahmad and Y. Zhou, Nonlocal fractional evolution inclusions of order $\alpha\in(1,2)$, \emph{Mathematics}, {\bf 209} (2019), 1–17 .
		
		\bibitem{EHD} E. Hernández and D. O’Regan, On a new class of abstract impulsive differential equations, \emph{Proc. Amer. Math. Soc.} {\bf 141} (2013), 1641–1649.
		
		\bibitem{HlR} R. Hilfer, \emph{Applications of Fractional Calculus in Physics}, World Scientific: Singapore, 2000. 
		
		\bibitem{HY}  Y. Hino, S. Murakami and T. Naito, \emph{Functional Differential Equations with Infinite Delay}, Lecture Notes in Mathematics, Springer-Verlag, Berlin, 1991.
		
		
		\bibitem{AAK} A.A. Kilbas, H.M. Srivastava and J.J. Trujillo, \emph{Theory and Applications of Fractional Differential Equations, in: North-Holland Mathematics Studies}, Elsevier Science B.V., Amsterdam, 2006.
		
		\bibitem{JKI} J. Kisyński, On cosine operator functions and one parameter group of operators, \emph{Studia Math.}, {\bf 49} (1972), 93–105.
		
		\bibitem{SNS} S. Kumar and N. Sukavanam, Approximate controllability of fractional order semilinear systems with bounded delay, \emph{J. Differential Equations}, {\bf 252} (2012),  6163–6174.
		
		\bibitem{Jkr} J. Klafter and R. Metzler, The random walk’s guide to anomalous diffusion: a fractional dynamics approach, \emph{Phys. Rep.} {\bf 339} (2000), 1-77.
		
		\bibitem{Jkj}J. Kemppainen, J. Siljander, V. Vergara and R. Zacher, Decay estimates for time fractional and other non-local in time subdiffusion equations in $\mathbb{R}^d$, \emph{Math. Annalen} {\bf 366} (2016) 941-979.
		
		\bibitem{LVB} V. Lakshmikantham, D. D. Bainov and P.S. Simeonov, {\em Theory of Impulsive Differential Equations}, World Scientific, Singapore (1989).
		
		\bibitem{JYONG}	X. Li. and J. Yong, \emph{Optimal Control Theory for Infinite Dimensional Systems}, Birkh\"auser Basel, 1995.
		
		\bibitem{LA} A. Lunardi, On the linear heat equation with fading memory, {\em SIAM J. Math. Anal.}, {\bf 21} (1990), 1213-1224.
		
		\bibitem{MIN}  N.I. Mahmudov, Approximate controllability of fractional neutral evolution equations in Banach spaces, \emph{Abstr. Appl. Anal.}, {\bf 2013} (2013), 531894.
		
		\bibitem{NMI} N.I. Mahmudov, Approximate controllability of fractional Sobolev-type evolution equations in Banach spaces, \emph{Abstr. Appl. Anal.}, {\bf 2013} (2013), 502839.
		
		\bibitem{M} N.I. Mahmudov, Approximate controllability of semilinear deterministic and stochastic evolution equations in abstract spaces, \emph{SIAM J. Control Optim.}, {\bf 42} (2003), 1604-1622.
		
		
		\bibitem{Fm}F. Mainardi, \emph{Fractional Calculus and Waves in Linear Viscoelasticity: An Introduction to Mathematical Models}, Imperial College Press, 2010.
		
		\bibitem{NE} D. Nesic and A.R. Teel, Input-to-state stability of networked control systems, \emph{Automatica} {\bf 40} (2004), 2121–2128.
		
		
		\bibitem{NJ} J.W. Nunziato, On heat conduction in materials with memory, {\em Quart. Appl. Math.}, {\bf 29} (1971), 187-204.
		
		\bibitem{VOj} V. Obukhovski and J.C. Yao, On impulsive functional differential inclusions with Hille-Yosida operators in Banach spaces, \emph{Nonlinear Anal.}, {\bf 73} (2010), 1715-1728.
		
		\bibitem{P} {\sc A. Pazy}, \emph{Semigroup of Linear operators and Applications to partial equations}, Springer-Verlag, New York, 1983.
		
		\bibitem{Ip} I. Podlubny,\emph{ Fractional Differential Equations}, Academic Press, San Diego, 1999.
		
		
		\bibitem{Qh} H. Qin, X. Zuo, J. Liu and L. Liu, Approximate controllability and optimal controls of fractional dynamical systems of order $1 < q < 2$ in Banach spaces, \emph{Adv. Differ. Equ.}, {\bf 2015} (2015) 1–17.
		
		\bibitem{Mvv} M.M. Raja and V. Vijayakumar, New results concerning to approximate controllability of fractional integro-differential evolution equations of order $1<r<2$, \emph{Numer. Methods Partial Differ. Eq.}, (2020), 1-16, https://doi.org/10.1002/num.22653.
		
		
		\bibitem{MMr} M.M. Raja, V. Vijayakumar, and R. Udhayakumar, Results on the existence and controllability of fractional integro-differential system of order $1< r< 2$ via measure of noncompactness, \emph{Chaos, Solitons \& Fractals}, {\bf 139} (2020): 110299.
		
		\bibitem{MMv} M.M. Raja, V. Vijayakumar, R. Udhayakumar and Y. Zhou, A new approach on the approximate controllability of fractional differential evolution equations of order $1 < r < 2$ in Hilbert spaces, \emph{Chaos, Solitons \& Fractals}, {\bf 141} (2020): 110310.
		
		
		\bibitem{YAR} Y.A. Rossikhin and M.V. Shitikova, Application of fractional derivatives to the analysis of damped vibrations of viscoelastic single mass systems, \emph{Acta Mech.}, {\bf 120} (1997), 109–125.
		
		\bibitem{IPF} S.G. Samko, A.A. Kilbas and O.I. Marichev, \emph{Fractional Integrals and Derivatives}, Gordon and Breach Science Publishers, Switzerland, 1993.
		
		\bibitem{ER} R. Sakthivel and E.R. Anandhi, Approximate controllability of impulsive differential equations with state-dependent delay, {\em Internat. J. Control}, {\bf 83} (2010), 387-393.
		
		
		\bibitem{Sr} R. Sakthivel, R. Ganesh, Y. Ren and S.M. Anthoni, Approximate controllability of nonlinear fractional dynamical systems, \emph{Commun. Nonlinear Sci. Numer. Simul.}, {\bf 18} (2013), 3498–508.
		
		
		\bibitem{IPF} S.G. Samko, A.A. Kilbas and O.I. Marichev, \emph{Fractional Integrals and Derivatives}, Gordon and Breach Science Publishers, Switzerland, 1993.
		
		\bibitem{Sbl} L. Shu, X.B. Shu and J. Mao, Approximate controllability and existence of mild solutions for Riemann–Liouville fractional stochastic evolution equations with nonlocal conditions of order $1 < \alpha< 2$, \emph{Fract. Calculus Appl. Anal.}, {\bf 22} (2019) 1086–1112 .
		
		\bibitem{Sb} X.B. Shu and Q. Wang, The existence and uniqueness of mild solutions for fractional differential equations with nonlocal conditions of order $1 <\alpha< 2$, \emph{Comput. Math. Appl.}, {\bf 64} (2012), 2100–2110.
		
		
		\bibitem{SS} S. Singh, S. Arora, M. T. Mohan and J. dabas, Approximate controllability  of second order impulsive systems with state-dependent delay in Banach spaces, {\em Evol. Equ. Control Theory}, (2020), doi: 10.3934/eect.2020103.
		
		\bibitem{MST} M.S. Tavazoei, M. Haeri, S. Jafari, S. Bolouki and M. Siami, Some applications of fractional calculus in suppression of chaotic oscillations, \emph{IEEE Trans. Ind. Electron.}, {\bf 55} (2008), 4094–4101.
		
		\bibitem{TCC1} C.C. Travis and G.F. Webb, Cosine families and abstract nonlinear second order differential equations, {\em Acta Math. Hungar.}, {\bf 32} (1978), 75-96.
		
		\bibitem{TCC2} C.C. Travis and G.F. Webb, Second order differential equations in Banach space, in \emph{Nonlinear Equations in Abstract Spaces}, Academic Press, 1978, 331-361.
		
		
		\bibitem{TRR} R. Triggiani, Addendum: A note on the lack of exact controllability for mild solutions in Banach spaces, \emph{SIAM J. Control Optim.}, {\bf 18} (1980), 98.
		
		\bibitem{TR} R. Triggiani, A note on the lack of exact controllability for mild solutions in Banach spaces,\emph{SIAM J. Control Optim.}, {\bf 15} (1977), 407-411.
		
		\bibitem{JRW} J.R. Wang and M. Fečkan, A general class of impulsive evolution equations, \emph{Topol. Methods Nonlinear Anal.} {\bf 46} (2015), 915–933.
		
		\bibitem{JMF} J.R. Wang and M. Fečkan, \emph{Non-Instantaneous Impulsive Differential Equations}, IOP, 2018.
		
		\bibitem{JWY} J. Wang and Y. Zhou, Existence and controllability results for fractional semilinear differential inclusions, \emph{Nonlinear Anal. Real. World Appl.}, {\bf 12} (2011), 3642–3653.
		
		\bibitem{JRY} J.R. Wang, Y. Zhou and  Z. Lin, On a new class of impulsive fractional differential equations, \emph{Appl. Math. Comput.} {\bf 242} (2014), 649–657.
		
		\bibitem{YTJ} Y. Tian, J.R. Wang and Y. Zhou, Almost periodic solutions for a class of non-instantaneous impulsive differential equations, \emph{Quaest. Math.} {\bf 42} (2019),  885–905.
		
		
		\bibitem{Vcr} V. Vijayakumar, C. Ravichandran, K.S. Nisar and K.D. Kucche, New discussion on approximate controllability results for fractional Sobolev type Volterra-Fredholm integro-differential systems of order $1<r <2$, \emph{Numer. Methods Partial Differ. Eq.}, (2021), 1–19, https://doi.org/10.1002/num.22772.
		
		\bibitem{TY} T. Yang and  L.O. Chua, Impulsive control and synchronization of nonlinear dynamical systems and application to secure communication, \emph{Internat. J. Bifur. Chaos Appl. Sci. Engrg.} {\bf 7} (1997), 645–664.
		
		\bibitem{Z} Z. Yan, Approximate controllability of partial neutral functional differential systems of fractional order with state-dependent delay, {\em Internat. J. Control}, {\bf 85} (2012), 1051-1062.
		
		
		\bibitem{Ch} C.C. Yeh, Discrete inequalities of the Gronwall-Bellman type in $n$ independent variables, \emph{J. Math. Anal. Appl.}, {\bf 105} (1985), 322-332.
		
		\bibitem{Yz} Y. Zhou and J.H. He, New results on controllability of fractional evolution systems with order $\alpha\in(1,2)$, \emph{Evol. Equ. Control Theory}, (2020),  \url{https://doi.org/10.3934/eect.2020077}.
		
		\bibitem{EZ} E. Zuazua, \emph{Controllability and observability of partial differential equations: some results and open problems}, in Handbook of differential equations: evolutionary equations, {\bf 3} (2007), 527-621.
	\end{thebibliography}
\end{document}